\documentclass[11pt,a4paper]{amsart}
\usepackage{fullpage}
\usepackage{amsmath,amssymb}
\usepackage{color}
\newtheorem{proposition}{Proposition}
\newtheorem{theorem}{Theorem}[section]

\newtheorem{lemma}[theorem]{Lemma}
{\theoremstyle{definition}
\newtheorem{definition}{Definition}[section]
 \newtheorem{hyp}{Hypothesis}
 }
 {\theoremstyle{remark}
\newtheorem{remark}{Remark}
 
 }

\newcommand{\vect}{\operatorname{span}}

\def\c1var{{C^{1-\text{var}}}}

\definecolor{cv}{rgb}{0.0, 0.42, 0.24}
\definecolor{by}{rgb}{0.74, 0.2, 0.64}
\usepackage[colorlinks=true,citecolor=red,linkcolor=blue,pdfpagetransition=Blinds]{hyperref} %si quito este paquete puedo ver que ecuaciones ya están citadas en los textos
\begin{document}
%%%%%%%%%%%%%%%%%%%%%%%%%%%%%%%%%%%%%%%%%%
%%%%%%%%%%%%%%%%%%%%%%%%%%%%%%%%%%%%%%%%%%%%%%%%%%%%%%%%%%%%%%%%%%%%%%%%%%%%%%%%%%%%%%%%%%%%%%%%%%%
%%%%%%%%%%%%%%%T I T U L O %%%%%%%%%%%%%%%%%%%%%%%%%%%%%%%%%%%%%%%%%%%%%%%%%%%%%%%%%%%%%%%%%%%%%%%%
%%%%%%%%%%%%%%%%%%%%%%%%%%%%%%%%%%%%%%%%%%%%%%%%%%%%%%%%%%%%%%%%%%%%%%%%%%%%%%%%%%%%%%%%%%%%%%%%%%%
\title{{\bf Rates of convergence for the number of zeros of random trigonometric polynomials}}

\author{Laure Coutin}
\address{Institut de math\'ematiques de Toulouse, Universit\'e Paul Sabatier,
118, route de Narbonne F-31062 Toulouse cedex 9, France}
\email{\texttt{laure.coutin@math.univ-toulouse.fr}}
\author{Liliana Peralta}
\address{Centro de Investigaci\'on en Matem\'aticas, UAEH,
Carretera Pachuca-Tulancingo km 4.5 Pachuca,
Hidalgo 42184, Mexico}
\email{\texttt{liliana\_peralta@uaeh.edu.mx}}
\thanks{The second author would like to thank the all members to the Institut de Mathématiques de Toulouse, France for their kind hospitality during her research stay. She also thanks the ``Programa de
Movilidad Educativa'' of the Direcci\'on de Relaciones
Internacionales e Intercambio Acad\'emico of UAEH, Mexico for its support to obtain a permission to carry out this stay.}

%%%%%%%%%%%%%%%%%%%%%%%%%%%%%%%%%%%%%%%%%%%%%%%%%%%%%%%%%%%%%%%%%%%%%%%%%%%%%%%%%
%%%%%%%%%%%%%%%%%%%%%%%%%%%%%%%%%%%%%%%%%%%%%%%%%%
%%%%%%%%%%%%%%%%%%%% A b s t r a c t  %%%%%%%%%%%%%%%%%%%%%%%%%%%%%%%%%%%%%%%%%%%%%%%%%%%
\begin{abstract}
In this paper, we quantify the rate of convergence between the distribution of number of zeros of random trigonometric polynomials (RTP) with i.i.d. centered random coefficients and the number of zeros of a stationary centered Gaussian process $G$, whose covariance function is given by the sinc function. First, we find the convergence of the RTP towards $G$ in the Wasserstein$-1$ distance, which in turn is a consequence of Donsker Theorem. Then, we use this result to derive the rate of convergence between their respective number of zeros.  Since the number of real zeros of the RTP is not a continuous function, we use the Kac-Rice formula to express it as the limit of an integral and, in this way, we approximate it by locally Lipschitz continuous functions. 
%\
\end{abstract}

%\normalsize

%%%%%%%%%%%%%%%%%%%%%%%%%%%%%%%%%%%%%%%%%%%%%%%%%%
%%%%%%%%%%%%%%%%%%%% K E  Y W O R D S %%%%%%%%%%%%%%%%%%%%%%%%%%%%%%%%%%%%%%%%%%%%%%%%%%%
%%%%%%%%%%%%%%%%%%%%%%%%%%%%%%%%%%%%%%%%%%%%%%%%%%%%

%%%%%%%%%%%%%%%%%%%%%%%%%%%%%%%%%%%%%%%%%%%%%%%%%%
%%%%%%%%%%%%%%%%%%%% A M S c o d e s%%%%%%%%%%%%%%%%%%%%%%%%%%%%%%%%%%%%%%%%%%%%%%%%%%%
%%%%%%%%%%%%%%%%%%%%%%%%%%%%%%%%%%%%%%%%%%%%%%%%%%%%

%\AMScodes{42A05, 26C10, 60G15, 60G50}

%\small\tableofcontents

\maketitle

%%%%%%%%%%%%%%%%%%%%%%%%%%%%%%%%%%%%%%%%%%%%%%%%%%%%%%%%%%%%%%%%%%%%%%%%%%%%%%%%%
%%%%%%%%%%%%%%%%%%%%%%%%%%%%%%%%%%%%%%%%%%%%%%%%%%
%%%%%%%%%%%%%%%%%%%% I n t r o d u c t i o n   %%%%%%%%%%%%%%%%%%%%%%%%%%%%%%%%%%%%%%%%%%%%%%%%%%%
%%%%%%%%%%%%%%%%%%%%%%%%%%%%%%%%%%%%%%%%%%%%%%%%%%%%
%%%%%%%%%%%%%%%%%%%%%%%%%%%%%%%%%%%%%%%%%%%%%%%%%%%%%%%%%%%%%%%%%%%%%%%%%%%%%%%%%
\section{Introduction}
The behavior of zeros of random polynomials has been studied since the middle of last century (see for instance \cite{erdos}, \cite{kac}, \cite{littlewood}). This topic is relevant to the theory of Probability and other areas of science, since it is at the intersection of several branches of Mathematics and Physics such as, Linear Algebra, Number Theory, Mechanics, among others. Within the class of random polynomials, of particular interest are the trigonometric ones which have applications in Nuclear Physics. The study of these polynomials, together with the theory of random matrices, have had a great impact on the study of quantum chaotic dynamics and disordered systems \cite{Bogo}.

The random trigonometric polynomials were studied for the first time by Dunnage in \cite{Dunna}. In this work, it was proved that the mean number of real zeros of this class of polynomials with Gaussian coefficients is asymptotically proportional to the degree of the corresponding polynomial. Since then, different authors have studied the roots of these polynomials, see for instance \cite{Dunnage}, \cite{Farahmand87}, \cite{Farahmand95}, \cite{FarahSamba} and references therein. For example, in \cite{Azais-Leon} and \cite{Farahmand97}, we can find complementary information regarding the distribution of the number of zeros of this class of polynomials and more recently in the works \cite{FlasHen} and \cite{angst}, the universality of Dunnage's result has been proved, i.e., the result proved in \cite{Dunna} is also valid for polynomials with more general coefficients.

In this work we are interested in random trigonometric polynomials defined as follows 
\begin{equation}
\label{pol*}
X_m(t)=\frac{1}{\sqrt{m}}\sum_{r=0}^{m-1}\left[ x_r\cos\left(\frac{\pi r t}{m}\right)+ y_r \sin\left(\frac{\pi r t}{m}\right)\right],\quad t\in(0,1)
\end{equation}
for $m\in\mathbb{N}$ and where the coefficients $\{(x_r,y_r), r=0, 1, \ldots\}$ are a sequence of independent and identically distributed (i.i.d.) with zero mean random vectors defined on a probability space $(\Omega, \mathcal{F}, P)$ and identity covariance matrix. 

In \cite{Iks2016}, Iksanov et al. proved in distribution, in a suitable space of analytical functions, that $\{X_m\}_{m}$ converges to a  stationary and centered Gaussian process $G$ with covariance function given by
\begin{equation*}\label{def-r_G*}
r_{G}(s,t)=\sin_c(\pi (t-s)), \text{ for }s,t\in(0,1),
\end{equation*}
where $\sin_{c}(x):=\frac{\sin(x)}{x}$ if $x\neq 0$ and $\sin_{c}(0):=1$ by convention. 
For all real numbers $a<b,$ they also prove that the number of real zeros of $X_m$ on $[a,b]$ converges to the number of real zeros of $G$ in $[a,b].$ The limit distribution does not depend on the distribution of the coefficients $(x_1,y_1)$, a phenomenon referred
to as local universality. This result was first conjectured by Aza\"{\i}s et al. in \cite{JMA-Leon}, where they proved their conjecture assuming that $(x_1,y_1)$  has an infinitely smooth density that satisfies certain integrability conditions. 

In view of these results, a natural question would be whether it is possible to find rates of convergence between the distribution of the number of zeros of the polynomials defined in \eqref{pol*} and the number of zeros of the Gaussian process $G$. In this work, our aim is to quantify this rate of  convergence.

With this goal in mind one can first observe that a Donsker Theorem is hidden in this context. 
There exist the linear functions $\Theta_m$ and  $\Theta$ from $C([0,1],{\mathbb R}^2)$ into $C^2([0,1],{\mathbb R})$
(see Lemma \ref{rep}), such that 
\begin{align*}
&\Theta(B^1,B^2)= G,\\
&\Theta_m(S^m)=X_m
\end{align*}
where $(B^1,B^2)$ is a two dimensional Brownian motion,
\begin{align*} 
S^m(t) =\sum_{k=0}^{m-1} (x_k+ iy_k ) h_k^m(t),~~t \in [0,1]
\end{align*}
and $h_k^m(t)= \sqrt{m}\int_0^t {\mathbf 1}_{[\frac{k}{m},\frac{k-1}{m}]}(s) ds,~~t\in [0,1],~~k=0,\ldots ,m-1.$

There are several results for the rate of convergence in the Donsker Theorem, see for instance  \cite{Barbour90}, \cite{CoutinDecre}, \cite{CouDec2020} or the references therein. In \cite{CouDec2020}, the rate of convergence for the Donsker Theorem is obtained for the Kantorovitch-Rubinstein metric.

For a complete, separable metric space $W,$ the so-called Kantorovitch-Rubinstein or
Wasserstein-1 distance between $\mu$ and $\nu$  two probability measures on $W$ is defined by~:
\begin{align*}
dist_{KR}(\mu,\nu)= \sup_{f \in {\mathcal L}_{1}(W)} \int_W fd\mu, \int_W fd_nu
\end{align*}
where
\begin{align*}
{\mathcal L}_1(W)= \{f~: W\rightarrow {\mathbb R}, ~~~|f(x) -f(y)| \leq \text{dist}_W(x,y) ,~~\forall ~x,~~y ~~\in W\},
\end{align*}
The Fortet-Mourier distance between $\mu$ and $\nu$  two probability measures on $W$ is defined by~:
\begin{align*}
dist_{FM}(\mu,\nu)= \sup_{f \in {\mathcal L}_{1,\infty}(W)} \int_W fd\mu, \int_W fd_nu
\end{align*}
where ${\mathcal L}_{1,\infty}(W)$ is the set of elements of ${\mathcal L}_1(W)$ bounded by 1.
First, using Theorem 3.4 of \cite{CouDec2020}, if $\{(x_r,y_r), r=0, 1, \ldots\}$ are i.i.d. centered random vectors with identity covariance matrix and  having a finite  third moment yields
\begin{align*}
\sup_{F \in {\mathcal L}_{1}(C^2([0,1],{\mathbb R}))} {\mathbb E}(F(X_m)) - {\mathbb E}( F(G)) \leq c \|(x_1,y_1)\|^3_{L^3}\frac{\ln m}{m^{\frac{1}{6}}}
\end{align*}
where $c$ is a universal positive constant, see Proposition \ref{Donsker-loc}.

The number of zeros of a continuous function $f$ on a interval $[a,b],$ $N(f,[a,b])$, is not a continuous function of $f.$ Nevertheless, using the Kac-Rice formula which expresses the number of real zeros of a smooth function as a limit of an integral (see for instance \cite{Wse-Aza}, \cite{kac}) it can be approximated by a locally Lipschitz continuous function of $f$ (see Appendix \ref{kac}). We then derive the following result, presented in Theorem  \ref{Z_Rate} stated for the Fortet-Mourier distance. 

For all $\alpha <\frac{1}{24}$ $ a <b$ there exists a constant $c_{\alpha,a,b}$ such that for all sequence $\{(x_r,y_r), r=0, 1, \ldots\}$ of i.i.d  centered random vectors with identity covariance matrix and $4-$integrable,  it is satisfied that
 \begin{align*}
\sup_{F \in {\mathcal L}_{1,\infty}({\mathbb R})}{\mathbb E} \left( F(N(X_m,[a,b]))\right) - {\mathbb E} \left( F(N(G,[0,1]))\right)\leq c_{\alpha,a,b}{\mathbb E}\left(|x_0|^4 +|y_0|^4\right) m^{- \alpha} .
\end{align*}

As far as the authors knowledge, this result is new in the literature. It constitutes the first explicit rate of convergence between the zeros of the polynomials defined in \eqref{pol*} and the Gaussian process $G$. In addition, we further extend the above result to the regular case, i.e., assuming that $(x_1,y_1)$  has an infinitely smooth density that satisfies certain integrability conditions, see Hypothesis \ref{H2}.

The rest of the paper is organized as follows. In Sections \ref{FMR} we establish the framework and give a more precise statements of our main results. To achieve our goals, in Section \ref{RCTP} together with Appendix \ref{stein+}, we provide the necessary tools in order to prove a rate of convergence for the polynomials defined in \eqref{pol*} and the process $G$ for a class of test functions that we introduce and name linearly locally Lipschitz continuos functions (see Definition \ref{def-local-lip} and Theorem \ref{RTCL-loc}). The proof of the main result of the paper, i.e., the rate of convergence of the law of the number of zeros of \eqref{pol*} towards  the law of the number of  zeros of $G$ is presented in Section \ref{NR-case} and it requires the results proved in Section \ref{rc_nofzero}.

%%%%%%%%%%%%%%%%%%%%%%%%%%%%%%%%%%%%%%%%%%%%%%%%%%%%%%%%%%%%%%%%%%%%%%%%%%%%%%%%%
%%%%%%%%%%%%%%%%%%%%%%%%%%%%%%%%%%%%%%%%%%%%%%%%%%

\section{Framework and main results}\label{FMR}
In this section we introduce the basic notation and terminology which will be used throughout this paper. In addition, the main results of this work are summarized here.
%%%%%%%%%%%%%%%%%%%%%%%%%%%%%%%%%%%%%%%%%%%%%%%%%%%%%%%%%%
%%%%%%%%%%%%%%% p o l y n o m i a l s %%%%%%%%%%%%%%%%%%%%%%%%%%%%%%%%%%
%%%%%%%%%%%%%%%%%%%%%%%%%%%%%%%%%%%%%%%%%%%%%%%%%%%%%%%%%%

\subsection{Random trigonometric polynomials}

Let $\{(x_r,y_r), r=0, 1, \ldots\}$ be a sequence of i.i.d. random vectors defined on a probability space $(\Omega, \mathcal{F}, \mathbb{P})$ with zero mean and identity covariance matrix. For $m\in\mathbb{N}$, we consider the random trigonometric polynomials defined as follows 
\begin{equation}
\label{pol}
X_m(t)=\frac{1}{\sqrt{m}}\sum_{r=0}^{m-1}\left[ x_r\cos\left(\frac{\pi r t}{m}\right)+ y_r \sin\left(\frac{\pi r t}{m}\right)\right],\quad t\in(0,1).
\end{equation}

Let $(B^1,B^2)$ be a standard two-dimensional Brownian motion and $G$ be the centered Gaussian process defined as follows
\begin{align}\notag
G(t)&= \int_0^1 \cos (\pi tu) dB^{1}_u + \int_0^1 \sin (\pi tu) dB^{2}_u\\\label{def-lim-G}
&=\text{Re}\left[ \int_0^1 e^{-i\pi tu} \left[ dB^{1}_u + i dB^{2}_u \right] \right],~~t\in [0,1].
\end{align}

Then, the covariance function of $G$ is given by 
\begin{equation*}\label{def-r_G}
r_{G}(s,t)=\sin_c(\pi (t-s)), \text{ for }s,t\in(0,+\infty),
\end{equation*}
where $\sin_{c}(x):=\frac{\sin(x)}{x}$ if $x\neq 0$ and $\sin_{c}(0):=1$ by convention. 

The sequence of random polynomials $\{X_m\}_{m}$ defined in \eqref{pol} converges weakly with respect to the $C^1-$topology on $[ a,b]$ to the stationary Gaussian process $G$ on the space of analytic functions with values on $\mathbb{R}$ (cf. \cite[Theorem 3.1]{Iks2016}). In addition, this result is independent of the common law of the sequence $\{(x_r,y_r), r=0, 1, \ldots\}$. In consequence, in the first part of the paper, we seek to quantify the convergence of $\{X_m\}_{m}$ towards $G$ in the Wasserstein-$1$ distance on $C^1([a,b],{\mathbb R}).$

%~~\\
%{\bf LAURE
%Here I wonder if we should work in 
%\begin{enumerate}
%\item  the space of $C^1([0,1])$ functions endowed with the norm
%$$\|f\|_{\infty,1}=|f(0)| + \|\dot{f}\|_{\infty},$$
%\item or the space of analytical functions with reals values as in \cite{Iks2016}.
%\end{enumerate}
%}

%%%%%%%%%%%%%%%%%%%%%%%%%%%%%%%%%%%%%%%%%%%%%%%%%%%%%%%%%%
%%%%%%%%%%%%%%%  Convergence rate in Wasserstein distance %%%%%%%%%%%%%%%%%%%%%%%%%%%%%%%%%%
%%%%%%%%%%%%%%%%%%%%%%%%%%%%%%%%%%%%%%%%%%%%%%%%%%%%%%%%%%
\subsection{Convergence rate in Wasserstein distance}\label{Wdist}

The Wasserstein-$1$ distance (see \cite{villani}) between two probability measures on a complete separable Banach space $(E,|.|)$, $\mu$ and $\nu,$ is defined as follows

\begin{equation*}
\label{w1}
d_{\mathcal{W}}(\mu,\nu)=\sup_{F\in\mathcal{L}_{1,\infty}(E)}\left|\int F(x)d\mu(x)-\int F(x)d\nu(x)\right|,
\end{equation*}
where $\mathcal{L}_1(E)$ is the set of Lipschitz continuous functions on $ E$ with real values and Lipschitz constant at most 1, that is to say
\begin{equation*}
\mathcal{L}_1(E)=\left\{F:E\to \mathbb{R}\;\big| |F(x)-F(y)|\leq |x-y|, \text{ for any }x,y \in E\right \},
\end{equation*}
and  $\mathcal{L}_{1,\infty}(E)$ is the subset of functions $F\in{\mathcal L}_1(E)$ with values in $[-1,1]$ i.e., bounded by $1$. 
The space $C^i([0,1],{\mathbb R})$  of continuous functions with $i$ derivatives for $i=1,2$ is endowed with the norm $\|\cdot\|_{\infty,i}$ defined as follows
\begin{align}\label{def-norme}
\|f\|_{\infty,1}= \|\dot{f}\|_{\infty} + |f(0)|,~~\forall f\in C^1([0,1],{\mathbb R})\\
\intertext{and}
\|f\|_{\infty,2}= \|\ddot{f}\|_{\infty} +|\dot{f}(0)|+ |f(0)|,~~\forall f\in C^2([0,1],{\mathbb R}).
\end{align}

For some applications about the number of zeros of $X_m,$ we introduce the following class of test functions.

\begin{definition} \label{def-local-lip}
A function $F: E\rightarrow {\mathbb R}$ is said to be \textit{linearly locally Lipschitz} continuous with constant 1 if
\begin{align*}
\left| F(x+h) - F(x)\right| \leq \left[ |x|+ |h| \right]|h| ~\forall ~~x,h \in E.
\end{align*}
The set of all linearly locally Lipschitz continuous functions  with constant 1 is denoted by ${\mathcal Lip}_{1,1}(E).$
\end{definition}

For this kind of test functions we have the following result. 

\begin{theorem}\label{RTCL-loc}
There exists a constant $C$ such that for all $\{(x_r,y_r), r=0, 1, \ldots\}$ sequence of i.i.d. centered and $4-$integrable random vectors with identity covariance matrix it is satisfied that
 \begin{align*}
 \sup_{F \in {\mathcal Lip}_{1,1}(C([0,1],{\mathbb R}))}  {\mathbb E}( F(X_m)) -{\mathbb E} (F(G))\leq {C {\mathbb E}(|x_0|^4+|y_0|^4)\frac{\ln {m}}{m^{\frac{1}{6}}}}.
 \end{align*}
\end{theorem} 

In order to get the proof of Theorem \ref{RTCL-loc} it will be necessary to present some previous results (see Appendix \ref{stein+}) that will be used in Section \ref{RCTP}.

In the second part of the paper, we derive from Theorem \ref{RTCL-loc}  some estimations on the rate of convergence of the law of the number of zeros of $X_m$ towards  the law of the number of  zeros of $G.$ To do this, we give the following definition.

Let $f\,:[t_1,t_2]\to \mathbb{R}$ be a real-valued function defined on the interval $[t_1,t_2]$ of the real line. We denote 
\begin{equation*}
Z\left(f,[t_1,t_2]\right)=\{t\in [t_1,t_2]\,: f(t)=0\}
\end{equation*}
the set of the roots of equation $f(t)=0$ in the interval $[t_1,t_2]$ and 
\begin{equation}\label{N}
N(f,[t_1,t_2])=\left|Z\left(f,[t_1,t_2]\right)\right|
\end{equation}
the number of roots of equation $f(t)=0$.

Then we enunciate the main results of this part of the paper.
\begin{theorem}\label{Z_Rate}
Let $0<\theta <1/24.$ There exists a constant $C_{\theta}$ such that for all sequence  $\{(x_r,y_r), r=0, 1, \ldots\}$ of i.i.d. centered random vectors, $4-$integrable, with identity covariance matrix
\begin{align*}
\sup_{F \in {\mathcal L}_{1,\infty}({\mathbb R})}{\mathbb E} \left( F(N(X_m,[0,1]))\right) - {\mathbb E} \left( F(N(G,[0,1 ]))\right)\leq C_{
\theta}{\mathbb E}\left(|x_0|^4 +|y_0|^4\right) m^{- \theta} .
\end{align*}
\end{theorem}

The proof of the above result is presented in Section \ref{NR-case} and it requires the results proved in Section \ref{rc_nofzero}.

Following \cite{JMA-Leon},
for the regular case we establish the next hypothesis. 

\begin{hyp}\label{H2} 
A random variable $\textbf{b}$ satisfies the Hypothesis \ref{H2} if admits a density $\rho\; : \mathbb{R}\to (0,\infty)$ of the form $\rho=e^{-\Psi}$ where $\Psi\in C^{\infty}(\mathbb{R})$ and whose derivatives satisfy
\begin{equation*}
\Psi^{(p)}\in\bigcap_{q\geq 1}L^q\left(e^{-\Psi(x)}dx\right),\quad p\geq 1.
\end{equation*}
\end{hyp}

Then, the rate of convergence for the regular case is the following. 

\begin{theorem}\label{Z_Rate_regular}
Let $0<\theta <1/24$ and assume that the  coefficients of the polynomials $\{X_m\}_m$ in \eqref{pol} satisfy the Hypothesis \ref{H2}.  Then, there exists a constant $C_{\theta}$ such that for all sequence  $\{(x_r,y_r), r=0, 1, \ldots\}$ of i.i.d. random vectors with zero mean, identity covariance matrix and $4-$integrable it is satisfied that
\begin{align*}
\sup_{F \in {\mathcal L}_{1}({\mathbb R})}{\mathbb E} \left( F(N(X_m,[0,1]))\right) - {\mathbb E} \left( F(N(G,[0,1]))\right)\leq C_{
\theta}{\mathbb E}\left(|x_0|^4 +|y_0|^4\right) m^{-\theta} .
\end{align*}
\end{theorem}

\begin{remark}
Notice that unlike the Theorem \ref{Z_Rate}, the test functions for the regular case do not have to be bounded by $1$.
\end{remark}

For sake for simplicity, we prove our results for $a=0$ and $b=1.$ The details are the same in the general case.
%%%%%%%%%%%%%%%%%%%%%%%%%%%%%%%%%%%%%%%%%%%%%%%%%%%%%%%%%% 
%%%%%%%%%%%%%%% PROCESSES CONVERGENCE %%%%%%%%%%%%%%%%%%%%%%%%%%%%%%%%%%%%%%%%%%%%%%%%%%%%%%%%%%%%%%%%%%%%%%%%%%%%%%%%%%%
%%%%%%%%%%%%%%%%%%%%%%%%%%%%%%%%%%%%%%%%%%%%%%%%%%%%%%%%%%

\section{Rate of convergence of trigonometric polynomials}\label{RCTP}
%%%%%%%%%%%%%%%%%%%%%%%%%%%%%%%%%%%%%%%%%%%%%%%%%%%%%%%%%%
%%%%%%%%%%%%%%% d o n s k e r %%%%%%%%%%%%%%%%%%%%%%%%%%%%%%%%%%%%%%%%%%%%%%%%%%%%%%%%%%%%%%%%%%%%%%%%%%%%%%%%%%%
%%%%%%%%%%%%%%%%%%%%%%%%%%%%%%%%%%%%%%%%%%%%%%%%%%%%%%%%%%
In this section, we provide the necessary tools in order to prove the result established in Theorem \ref{RTCL-loc}.
\subsection{Link with the Donsker Theorem}\label{link}
Let $\{(x_r,y_r), r=0, 1, \ldots\}$ be a sequence of i.i.d. random vectors defined on a probability space $(\Omega, \mathcal{F}, \mathbb{P})$ with zero mean and identity covariance matrix. For $m\in\mathbb{N}$ we define
\begin{align*}\label{empirical-process}
S^m= \sum_{k=0}^{m-1} \left( x_k + iy_k\right) h_k^m,
\end{align*}
where 
\begin{equation*}
h_k^m(t)=\sqrt{m} \int_0^t {\mathbf 1}_{\left[\frac km,\frac{k+1}{m}\right]}(u) du,\quad t\in [0,1],\quad k=0,...,m-1.
\end{equation*}
%%%%%%%%%%%%%%%%%%%%%%%%%%%%%%%%%%%%%%%%%%%%%%%%%%%%%%%%%
%%%%%%%%%%%%rep%%%%%%%%%%%%%%%%%%%%%%%%%%%%%%%%%
Let us introduce the two following linear maps from $C([0,1],{\mathbb C}) $ into $C^1([0,1],{\mathbb R})$
\begin{align}\label{def-theta}
\Theta(f)(t)= \text{Re}\left[ e^{-i\pi t}f(1) -f(0) +i\pi t\int_0^1 e^{-i\pi tu} f(u) du\right],\quad f\in C([0,1],{\mathbb C}),
\end{align}
and 
\begin{align}\label{def-tehta-m}
\Theta_m(f)(t)= \text{Re}\left[ 
e^{-i\pi \frac{m-1}{m}t}f(1)- e^{i\pi \frac{1}{m}t}f(0) -
\sum_{k=0}^m \left( e^{-i t\pi  \frac{k}{m}} - e^{-i t\pi  \frac{k-1}{m}}\right) f\left(\frac{k}{m}\right)
\right].
\end{align}

Using stochastic integration by parts and the Abel transformation (summation by parts) we obtain the following Lemma.

\begin{lemma}\label{rep}
We have
\begin{align*}
&G=\Theta(B^1+iB^2),\\
&X_m= \Theta_m\left(S^m\right).
\end{align*}
\end{lemma}

\begin{proof}
The first identity is a consequence of the representation of $G$ given in \eqref{def-lim-G} and the integration by parts.
%%%%%%%%%%%%%%%%%%%%%%%%%%%%%%%%%%%%%%%%%%%%%%%%%%%%%%%%%%%%%%%
On the other hand, note that for $k\in\{0,....,m-1\}$
\begin{align*}
\frac{x_k+iy_k}{\sqrt{m}}= S^m\left(\frac{k+1}{m}\right) -S^m\left(\frac{k}{m}\right),
\end{align*}
then the second  equality is a consequence of  the definition of $X_m$ given in \eqref{pol} and the Abel transformation.
\end{proof}

In the next lemma we establish some properties of the linear maps $\Theta$ and $\Theta_m$.

%%%%%%%%%%%%%%%%%%%%%%%%%%%%%%%%%%%%%
\begin{lemma}\label{lip-theta}
The linear maps $\Theta$ and $\Theta_m$ are continuous from $C([0,1],{\mathbb C}) $ into $C^2([0,1],{\mathbb R}) $. Moreover, there exist some universal positive constants $C$ and $C_i,~~i=1,2$ such that for all $f\in C([0,1],{\mathbb C})$ and $m \in {\mathbb N}$ it is satisfied 
\begin{description}
\item[(i)] $\left\| \Theta(f) \right\|_{\infty,i} \leq C_i\|f\|_{\infty}$,
\item[(ii)] $\left\| \Theta_m(f) \right\|_{\infty,i} \leq C_i\|f\|_{\infty},~~i=1,2$
\item[(iii)] For all $m$ and for all $f$ $\alpha-$H\"older continuous function from $[0,1]$ into ${\mathbb R}^2$
\begin{align*}
\left\| \Theta_m(f)-\Theta(f) \right\|_{\infty,1} \leq  \frac{C}{m^{\alpha}}\left[\|f\|_{\infty} +\|f\|_{Hol,\alpha}\right],
\end{align*}
where $\|f\|_{Hol,\alpha}:=\sup_{(u,v)\in [0,1]^2,~~u\neq v}\frac{|f(u) -f(v)|}{|u-v|^{\alpha}}.$
\end{description}
\end{lemma}

\begin{proof}
\textbf{Case (i)}. For $f\in C([0,1],{\mathbb C})$ the definition of map $\Theta$ in equation \eqref{def-theta} implies that
\begin{align*}
 \Theta(f) (0)=f(1)-f(0)
\end{align*}
and
\begin{align*}
 \dot{\Theta}(f)(t)= &\text{Re}\left[ -i \pi e^{-i\pi t}f(1)  +i\pi \int_0^1 e^{-i\pi tu} f(u) du\right.\\
& \left.+ \pi^2 t \int_0^1 ue^{-i\pi tu} f(u) du\right].
\end{align*}
Moreover,
\begin{align*}
 \ddot{\Theta}(f)(t)= &\text{Re}\left[ - \pi^2 e^{-i\pi t}f(1)  +2\pi^2 \int_0^1 e^{-i\pi tu}u f(u) du\right.\\
& \left.-i\pi^3 t \int_0^1 u^2e^{-i\pi tu} f(u) du  \right].
\end{align*}
Point (i) is a consequence of  the definition of $\|f\|_{\infty,1}$ given in \eqref{def-norme}.

%%%%%%%%%%%%%%%%%%%%%%%%

\textbf{Case (ii)}. For sake of concision, we only investigate the case $i=1.$ From equality \eqref{def-tehta-m} for $f\in C([0,1],{\mathbb C})$ we have
\begin{align*}
 \Theta_m(f) (0)=f(1)-f(0)
\end{align*}
 and
\begin{align*}
\dot{\Theta}_m(f)(t)&=\text{Re}\left[ 
-i\pi \frac{m-1}{m}e^{-i\pi \frac{m-1}{m}t}f(1)- i\pi \frac{1}{m}e^{i\pi \frac{1}{m}t}f(0) \right]\\
&\quad+\text{Re}\left[i\pi 
\sum_{k=0}^m \left\{\frac{k}{m}\left( e^{-i t\pi  \frac{k}{m}} - e^{-i t\pi  \frac{k-1}{m}}\right)+\frac{1}{m} e^{-i t\pi  \frac{k-1}{m}}\right\} f\left(\frac{k}{m}\right)
\right].
\end{align*}
The function $t \mapsto e^{-it}$ is Lipschitz continuous with Lipschitz constant equal to 1 and
$$\frac{1}{m}\sum_{k=0}^m \frac{k}{m} \leq \int_0^1 udu =\frac{1}{2}.$$
Therefore, point (ii) for $i=1$ is a consequence of  the definition of $\|f\|_{\infty,1}$ given in \eqref{def-norme}.

Similar computations yield the proof of point (ii) for $i=2.$

\textbf{Case (iii)}. We deduce that
\begin{align*}
\left(\dot{\Theta}_m(f)-\dot{\Theta}(f)\right)(t)&= Re\left[ 
-i\pi \left\{\frac{m-1}{m}e^{-i\pi \frac{m-1}{m}t}-e^{-i \pi t}\right\}f(1)- i\pi \frac{1}{m}e^{i\pi \frac{1}{m}t}f(0)  \right]\\
&+Re\left[i\pi \left\{
\sum_{k=0}^m \frac{k}{m}\left( e^{-i t\pi  \frac{k}{m}} - e^{-i t\pi  \frac{k-1}{m}}\right)f\left(\frac{k}{m}\right)\right.\right.\\ &\quad\quad\quad\quad\quad\left.\left.+i\pi t\int_0^1u e^{-i\pi tu} f(u) du \right\}
\right]\\
&+ Re\left[ i\pi \left\{\sum_{k=0}^m \frac{1}{m} e^{-i t\pi  \frac{k-1}{m}}f\left(\frac{k}{m}\right) -\int_0^1 e^{-i\pi tu} f(u) du\right\}\right].
\end{align*}
Assume that $f$ is $\alpha-$H\"older continuous, 
then $u\mapsto u f(u)$ is $\alpha-$H\"older continuous on $[0,1]$ and
\begin{align*}
&\left| \sum_{k=0}^m \frac{k}{m}\left( e^{-i t\pi  \frac{k}{m}} - e^{-i t\pi  \frac{k-1}{m}}\right)f\left(\frac{k}{m}\right) +i\pi t\int_0^1u e^{-i\pi tu} f(u) du \right| \\
~~~&\leq \pi\sum_{k=0}^m \int_{\frac{k}{m}}^{\frac{k+1}{m}}\left|e^{-it\pi u}\left(
\frac{k}{m}f\left(\frac{k}{m}\right) -uf(u)\right)\right|du\\
&\leq \pi\left[ \|f\|_{\infty}\frac{1}{m} + \|f\|_{Hol,\alpha}\frac{1}{m^{\alpha}}\right].
\end{align*}

Moreover, $u \mapsto e^{-it\pi u} $ is Lipschitz continuous then
\begin{align*}
\left|\sum_{k=0}^m \frac{1}{m} e^{-i t\pi  \frac{k-1}{m}}f\left(\frac{k}{m}\right) -\int_0^1 e^{-i\pi tu} f(u) du\right|\leq
\frac{2\pi}{m}\|f\|_{\infty} +  \|f\|_{Hol,\alpha}\frac{1}{m^{\alpha}}.
\end{align*}
Choosing $C= 4(1+\pi^2)$ the proof of point (iii) is achieved.

\end{proof} 

%%%%%%%%%%%%%%%%%%%%%%%%%%%%%%%%%%%%%%%%%%%%%%%
%%%%%%%%%%%%%%%%%%%%%%%%%%%%%%%%%%%%%%%%%%%%%%%%
%%%%%%%%%%%%%%%%%%%%%%%%%%%%%%%%%%%%%%%%%%%%%%%5
\subsection{ Donsker Theorem in Wasserstein 1 distance} \label{DTw_1}
First, we recall  Theorem 3.4 of \cite{CouDec2020} in our setting.

\begin{theorem}\label{RC-B}
 There exists a constant $C$ such that for all sequence $\{(x_r,y_r), r=0, 1, \ldots\}$ of i.i.d. centered and $3-$integrable random vectors with identity covariance matrix it is satisfied that
 \begin{align*}
 \sup_{F \in {\mathcal L}_{1}(C([0,1],{\mathbb R}))}  {\mathbb E}( F(S^m)) -{\mathbb E} (F(B^1+iB^2))\leq {C} {\mathbb E}(|x_0|^3+|y_0|^3)\frac{\ln {m}}{m^{\frac{1}{6}}}.
 \end{align*}
\end{theorem} 

A careful reading of the proof of  Theorem 3.4 of \cite{CouDec2020} yields the following Proposition.

\begin{proposition}\label{Donsker-loc}
 There exists a constant $C$ such that for all sequence $\{(x_r,y_r), r=0, 1, \ldots\}$ of i.i.d. centered and $4-$integrable random vectors with identity covariance matrix it is satisfied that
 \begin{align*}
 \sup_{F \in {\mathcal Lip}_{1,1}(C([0,1],{\mathbb R}))}  {\mathbb E}( F(S^m)) -{\mathbb E} (F(B^1+iB^2))\leq {C {\mathbb E}(|x_0|^4+|y_0|^4)\frac{\ln {m}}{m^{\frac{1}{6}}}}.
 \end{align*}
\end{proposition} 
Since the proof of Proposition \ref{Donsker-loc} is close to the proof of Theorem 3.4 of \cite{CouDec2020} it is presented in the  Appendix \ref{stein+}.

Now we are in position to proof Theorem \ref{RTCL-loc}.
%%%%%%%%%%%%%%%%%%%%%%%%%%%%%%%%%%%%%%%%%%%%
%%%%%%%%%%%%%%%%%%%%%%%%%%%%%%%%%%%%%%%%%%%%

\subsubsection*{Proof of Theorem \ref{RTCL-loc}}

Let $F\in {\mathcal Lip}_{1,1}(C^1([0,1],{\mathbb R}))$ and  $m\in {\mathbb N}.$
Note that for $(f_1,f_2),~~(h_1,h_2) \in C([0,1],{\mathbb R}^2)$
\begin{align*}
&\left| F\circ \Theta_m(f_1+h_1+i[f_2+h_2]) -F\circ \Theta_m(f_1+if_2)\right| \\
&\quad\leq \left[ \|\Theta_m(f_1+if_2)\|_{\infty,1} +\|\Theta_m(h_1+ih_2)\|_{\infty,1}\right]\|\Theta_m(h_1+ih_2)\|_{\infty,1}.
\end{align*}
Since $\Theta_m$ is linear continuous from $C([0,1],{\mathbb C})$ into $C^1([0,1],{\mathbb R})$ (see Lemma \ref{lip-theta}), there exists a universal positive constant $C$ such that for all $m$ 
\begin{align*}
\left| F\circ\Theta_m(f_1+h_1+i[f_2+h_2]) -F\circ\Theta_m(f_1+if_2)\right|\\\leq C^2 \left[ \sum_{i=1}^2\left(\|f_i\|_{\infty} +\|h_i\|_{\infty}\right)\right] \sum_{i=1}^2\|h_i\|_{\infty}.
\end{align*}
The Definition \ref{def-local-lip} implies that 
$\frac{1}{C}( F\circ \Theta_m )\in {\mathcal Lip}_{1,1}(C([0,1],{\mathbb R}^2)).$

%%%%%%%%%%%%%%%%%%%%%%%%%%%%%%%%%%%%%%%%
According to Lemma \ref{rep}, $ F(X_m)=F\circ \Theta_m(S_m)$ and $F(G)=F\circ \Theta(B^1+iB^2),$ then
\begin{align*}
{\mathbb E}\left(F(X_m)\right) -{\mathbb E}\left( F(G)\right)&={\mathbb E}\left(F\circ\Theta_m(S_m)\right) -{\mathbb E}\left( F\circ\Theta_m(B^1+iB^2)\right)\\
&\quad +
{\mathbb E}\left( \left[ F\circ \Theta_m(B^1+iB^2)-F\circ \Theta(B^1+iB^2)\right] \right).
\end{align*}

Let $1/2>\alpha >1/6,$ the sample paths of $B^1+iB^2$ are $\alpha-$H\"older Lipschitz continuous and
$\|B^1+iB^2\|_{\infty}\leq \|B^1+iB^2\|_{Hol,\alpha}$ belongs to $L^2(\Omega,\mathcal{F},{\mathbb P};{\mathbb R}).$

Using the fact that  $F \in {\mathcal Lip}_{1,1}(C^1([0,1],{\mathbb R}))$ and  Lemma \ref{lip-theta} there exists a constant $C$ such that for all $m\in {\mathbb N},$
\begin{align*}
{\mathbb E}&\left( \left[ F\circ \Theta_m(B^1+iB^2)-F\circ \Theta(B^1+iB^2)\right] \right)\\
&\leq {\mathbb E}\left( \|\Theta_m(B^1+iB^2)-\Theta(B^1+iB^2)\|_{\infty,1} \|\Theta(B^1+iB^2)\|_{\infty,1}\right)\\
&\quad+{\mathbb E}\left(\|\Theta(B^1+iB^2)\|_{\infty,1} \|\Theta(B^1+iB^2)\|_{\infty,1}\right)\\
&\leq \frac{C}{m^{\alpha}}{\mathbb E}\left(\left[\|B^1+iB^2\|_{\infty}+\|B^1+iB^2\|_{Hol,\alpha}\right]\|B^1+iB^2\|_{\infty}\right)\\
&\leq\frac{C}{m^{\alpha}}{\mathbb E} \left(\|B^1+iB^2\|_{Hol,\alpha}^2\right).
\end{align*}

Using Proposition \ref{Donsker-loc},  there exists a constant $C$ such that for all sequence of i.i.d. centered, $4-$integrable random vectors $\{(x_r,y_r), r=0, 1, \ldots\}$ with identity covariance matrix, we get
\begin{align*}
{\mathbb E}\left(F(X_m)\right) -{\mathbb E}\left( F(G)\right)\leq {C {\mathbb E}(|x_0|^4+|y_0|^4)\frac{\ln {m}}{m^{\frac{1}{6}}}}.
\end{align*}
Therefore, the Theorem \ref{RTCL-loc} is proved.

%%%%%%%%%%%%%%%%%%%%%%%%%%%%%%%%%%%%%%%%%%%%%%%%%%%%%%%%%%%%%%%%%%%%%%%%%%%%%%%
%%%%%%%%%%%%%%%%%RATE OF CONVERGENCE of NUMBER OF ZEROS
%%%%%%%%%%%%%%%%%%%%%%%%%%%%%%%%%%%%%%%%%%%%%%%%%%%%%%
\section{ Rate of convergence of number of zeros}\label{rc_nofzero}
%
%\subsection{General case}
In this section, we want to estimate the rate of convergence of the law of the number of zeros of $X_m$ towards  the law of the number of  zeros of $G.$ In other words, we prove Theorems \ref{Z_Rate} and \ref{Z_Rate_regular}.
 For this purpose we will split the difference
\begin{equation*}\label{Rc}
{\mathbb E}\left(F\left(N(X_m,[0,1])\right)\right)-{\mathbb E}\left(F\left(N(G, [0,1])\right)\right)
\end{equation*}
as 
\begin{align}\label{sumII}
&{\mathbb E}\left(F\left(N(X_m,[0,1])\right)\right)-{\mathbb E}\left(F\left(N(G,[0,1])\right)\right)= \sum_{i=1}^5I_i,
\end{align}
where
\begin{align*}
&I_1= {\mathbb E}\left(F\left(N(X_m,[0,1])\right)\right)-{\mathbb E}\left(F\left(\Phi_{\delta}(X_m)\right)\right)\\
&I_2={\mathbb E}\left(F\left(\Phi_{\delta}(X_m)\right)\right)-{\mathbb E}\left(F\left(\Phi_{\delta,\varepsilon}(X_m)\right)\right)\\
&I_3={\mathbb E}\left(F\left(\Phi_{\delta,\varepsilon}(X_m)\right)\right)-{\mathbb E}\left(F\left(\Phi_{\delta,\varepsilon}(G)\right)\right)\\
&I_4={\mathbb E}\left(F\left(\Phi_{\delta,\varepsilon}(G)\right)\right)-{\mathbb E}\left(F\left(\Phi_{\delta}(G)\right)\right)\\
&I_5={\mathbb E}\left(F\left(\Phi_{\delta}(G)\right)\right)-{\mathbb E}\left(F\left(N(G,[0,1])\right)\right).
\end{align*} 

Remember that $N(f,[0,1])$ is the number of roots of the function $f$ in the interval $[0,1]$ (see  equality \eqref{N} below). In addition, we have defined the functions $ \Phi_{\delta}$ and $  \Phi_{\delta,\varepsilon}$ as follows. For $f\;:[0,1]\to \mathbb{R}$ such that $f\in C^1([0,1])$ we have 
 \begin{align*}
 \Phi_{\delta}(f)&=\frac{1}{2\delta} \int_0^1|\dot{f}(u)| {\mathbf 1}_{|f(u)| \leq \delta} du,\\
  \Phi_{\delta,\varepsilon}(f)&=\frac{1}{2\delta} \int_0^1|\dot{f}(u)| H_{\delta,\varepsilon}(f(u)) du,
 \end{align*}
where $H_{\delta,\varepsilon}$ is the affine function by steps that takes the value 1 if $x\leq \delta$ and $0$ if $x\geq \delta +\varepsilon$, i.e.,
\begin{align*}
H_{\delta,\varepsilon}(u)=
\begin{cases}
1, &\textnormal{if } |u|\leq \delta, \\
0, &\textnormal{if } |u|\geq  \delta+\varepsilon, \\
-\frac{(|u|-\delta)}{\varepsilon} +1, &\textnormal{if } \delta \leq | u |\leq \delta +\varepsilon,
\end{cases}
\end{align*}
and
\begin{align*}
{\mathbf 1}_{|u|\leq \delta}-H_{\delta,\varepsilon}(u) = 
\begin{cases}
0 &\mbox{if } |u| \leq \delta, \\
\frac{(|u|-\delta)}{\varepsilon} - 1 &\mbox{if } \delta \leq |u|\leq \delta +\varepsilon, \\
0 &\mbox{otherwise }
\end{cases}
\end{align*}

To simplify the notation we set $N(f,[0,1])=N(f)$.
\subsection{Estimation of $I_3$}
\begin{lemma}\label{lemma-I3}
 There exists a constant $C$ such that for all $\{(x_r,y_r), r=0, 1, \ldots\}$ sequence of i.i.d centered and $4-$integrable random vectors with identity covariance matrix it is satisfied that
\begin{align*}
\sup_{ F \in {\mathcal L}_{1}({\mathbb R})}{\mathbb E} \left(F\circ \Phi_{\delta,\varepsilon}(X_m) \right) 
- {\mathbb E} \left(F\circ \Phi_{\delta,\varepsilon}(G) \right) \leq C {\mathbb E}(|x_0|^4 +|y_0|^4)\frac{\ln m}{m^{1/6}} \frac{1}{\delta}\left(1+ \frac{1}{\varepsilon}\right).
\end{align*}
\end{lemma}
Note that the set of test functions is ${\mathcal L}_1({\mathbb R}).$

\begin{proof}
Let $x,y \in C^1([0,1],{\mathbb R}).$ The function $H_{\delta,\varepsilon}$ is bounded by $1$ and $1/\varepsilon-$Lipschitz continuous on ${\mathbb R}.$ Then
\begin{align}\label{b-3}
\delta\varepsilon\left| \Phi_{\delta,\varepsilon}(x)-\Phi_{\delta,\varepsilon}(y)\right| \leq \varepsilon\|\dot{x} -\dot{y}\|_{\infty}+\|\dot{y}\|_{\infty} \|x-y\|_{\infty}.
\end{align}
and
\begin{equation*}\frac{C\delta \varepsilon}{\varepsilon +1}|\Phi_{\delta, \varepsilon}(x)-\Phi_{\delta, \varepsilon}(y)|\leq \left[\|y\|_{\infty,1}+\|x-y\|_{\infty,1}\right]\|x-y\|_{\infty,1}
\end{equation*}

In other words, $\frac{C\delta \varepsilon}{1+\varepsilon} \Phi_{\delta,\varepsilon}$ belongs to ${\mathcal Lip}_{1,1}(C^1([0,1],{\mathbb R}))$ and since $F$ is a $1-$Lipschitz continuous function, then $\frac{C\delta \varepsilon}{1+\varepsilon} F\circ\Phi_{\delta,\varepsilon}$ belongs to ${\mathcal Lip}_{1,1}(C^1([0,1],{\mathbb R})).$
Lemma \ref{lemma-I3} is a consequence of Proposition \ref{Donsker-loc}.
\end{proof}

\subsection{Estimation of $I_4$}
\begin{lemma}\label{lemma-I4}
There exists a positive constant $C$ verifying 
\begin{align*}
\sup_{ F \in {\mathcal L}_{1}({\mathbb R})}{\mathbb E} \left(F\circ \Phi_{\delta,\varepsilon}(G) \right) 
- {\mathbb E} \left(F\circ \Phi_{\delta}(G) \right) \leq C \frac{\varepsilon}{\delta}.
\end{align*}
\end{lemma}

Note that the set of test functions is ${\mathcal L}_1({\mathbb R}).$

\begin{proof}
Let $F$ be in ${\mathcal L}_1({\mathbb R})$ then
\begin{align*}
{\mathbb E} \left(F\circ\Phi_{\delta,\varepsilon}(G) \right) 
- {\mathbb E} \left(F\circ \Phi_{\delta}(G) \right) &\leq {\mathbb E}\left| \Phi_{\delta,\varepsilon}(G) -\Phi_{\delta}(G) \right|\\
&\leq \frac{1}{2\delta} \int_0^1 {\mathbb E} \left( |\dot{G}(s) | \frac{||G(s)| - \delta -\varepsilon|}{\varepsilon} {\mathbf 1}_{\delta \leq |G(s)| \leq \delta +\varepsilon} \right)ds.
\end{align*}
Using that the Gaussian vector $(G(s),\dot{G}(s))$ is centered with covariance matrix
\begin{equation}\label{cm_g}
\Gamma_s=\left(\begin{array}{cc} 1 &0\\
0 &\frac{\pi^2}{3}\end{array}\right) 
\end{equation} 
we get
\begin{align*}
{\mathbb E} \left(F\circ\Phi_{\delta,\varepsilon}(G) \right) 
- {\mathbb E} \left(F\circ \Phi_{\delta}(G) \right) 
&\leq \frac{1}{2\delta} \int_0^1 \int_{{\mathbb R}^2}|y|\left|\frac{|x|-\delta-\varepsilon}{\varepsilon}\right|\frac{\sqrt{3}}{2\pi^2}e^{-\frac{3y^2+\pi^2x^2}{2\pi^2}}dydxds\\
&\leq \frac{\sqrt{3}\varepsilon}{6\delta}.
\end{align*}
\end{proof}

\subsection{Estimation of $I_5$}

\begin{lemma}\label{lem-I5-X}
Let $X$ be a process with sample paths in $C^2([0,1],{\mathbb R})$ such that for all $0<\theta <1$ and $p\geq 1$ it is satisfied that
\begin{align}\label{hipo}
{\mathbb E} \left( \|X\|_{\infty,2}^p \right)+ \sup_{t\in [0,1]} {\mathbb E} \left( |\dot{X}(s)|^{-\theta} \right) <\infty.   
\end{align}
Then there exists a positive constant $C$ such that
\begin{align*}
\sup_{F \in {\mathcal L}_{1,\infty}({\mathbb R})} {\mathbb E} (F\circ N(X)) - {\mathbb E}\left( F \circ \Phi_{\delta} (X)\right) \leq C\delta^{\theta}.
\end{align*}
\end{lemma}

\begin{proof}
Let $\eta >0$ and $Y^{\theta}_{\eta}$ be defined as
\begin{align*}
Y^{\theta}_{\eta}(s)=\frac{ 1}{\left[\eta + \ X(s)^2+\dot{X}(s)^2 \right]^{\theta/2} },~~s\in [0,1].
\end{align*}
Since the sample paths of $X$ belong to $C^2([0,1],\mathbb{R})$ we have

\begin{align*}
Y^{\theta}_{\eta}(t)=Y^{\theta}_{\eta}(0)-\theta\int_0^t [X(s) + \ddot{X}(s)]\frac{ \dot{X}(s) }{\left[\eta + X(s)^2+\dot{X}(s)^2 \right]^{\theta/2+1}}ds
\end{align*}

Then,
\begin{align*}
\sup_{t \in [0,1]} |Y^{\theta}_{\eta}(t)| \leq |Y^{\theta}_{\eta}(0)|+\theta\sup_{t\in [0,1]} [|X(t)| + |\ddot{X}(t)|]
\int_0^1\frac{ 1 }{|\dot{X}(s)|^{\theta}}ds<\infty.
\end{align*}
Note that 
\begin{align*}
\sup_{t\in [0,1]} [|X(t)| + &|\ddot{X}(t)|]\leq 2 \|X\|_{\infty, 2}\\
\intertext{and}
|Y^{\theta}_{\eta}(0)|&\leq |\dot{X}(0)|^{-\theta}.
\end{align*}
Therefore using H\"older inequality with $p$ such that $p\theta <1$ we get
\begin{align*}
{\mathbb E} \left(\sup_{t \in [0,1]} |Y^{\theta}_{\eta}(t)| \right)
\leq {\mathbb E}(|\dot{X}(0)|^{-\theta})+2\theta{\mathbb E}\left(\|X\|_{\infty,2}^q \right)^{1/q}
{\mathbb E}\left(\int_0^1\frac{ 1 }{|\dot{X}(s)|^{p\theta}}\right)^{1/p}ds<\infty.
\end{align*}

Letting $\eta$ going to 0, we obtain
\begin{align}\label{maj-inv-theta}
{\mathbb E} \left(\sup_{t\in [0,1]} \frac{1}{\left(|X(t)|^2+|\dot{X}(t)|^2\right)^{\theta/2}}\right)
\leq  {\mathbb E}(|\dot{X}(0)|^{-\theta})+2\theta{\mathbb E}\left(\|X\|_{\infty,2}^q \right)^{1/q}
{\mathbb E}\left(\int_0^1\frac{ 1 }{|\dot{X}(s)|^{p\theta}}\right)^{1/p}ds.
\end{align}
In consequence, the sample paths of $X$ fulfill Hypothesis \ref{hyp-H1} almost surely. 

Set
\begin{align}\label{notation_A}
{\mathcal A}_{X}=\min \left\{ |X(0)|,|X(1)|,\frac{1}{2} \min_{t\in (0,1)}|X(t)|+|\dot{X}(t)|\right\}.
\end{align} 
By Lemma \ref{lem-n-delta-sta} we know that\\
\begin{align}\label{bound_A}
{\mathbb P}\left( N(X) \neq \Phi_{\delta}(X) \right)& \leq {\mathbb P}\left( {\mathcal A}_X \leq \delta \right)\\\notag
&\leq {\mathbb P}\left(|X(0)|^{-1}\geq \delta^{-1}\right) + {\mathbb P}\left(|X(1)|^{-1}\geq \delta^{-1}\right)\\\notag
~~~&+
{\mathbb P}\left(\sup_{t\in[0,1]}\frac{1}{|X(t)|+|\dot{X}(t)|}\geq \frac{\delta^{-1}}{2}\right)\\\notag
&\leq 3 {\mathbb P} \left( \sup_{t\in[0,1]}\frac{1}{|X(t)|+|\dot{X}(t)|}\geq \frac{\delta^{-1}}{2}\right).
\end{align}
And Markov property implies
\begin{align}\label{bound_A_2}
{\mathbb P}\left( N(X) \neq \Phi_{\delta}(X) \right) \leq 3(2\delta)^{\theta}
{\mathbb E} \left( \left|\sup_{t \in [0,1]} \frac{1}{\left[|X(t)| +|\dot{X}(t)|\right]}\right|^{\theta}\right).
\end{align}

Since $F$ is a $1-$Lipschitz continuous function bounded by $1$, it follows that
\begin{align*}
 {\mathbb E} (F\circ N(X)) - {\mathbb E}\left( F \circ \Phi_{\delta} (X)\right)& \leq 2 {\mathbb P}\left( N(X) \neq \Phi_{\delta}(X) \right)\\
&\leq 6(2\delta)^{\theta}
{\mathbb E} \left( \left|\sup_{t \in [0,1]} \frac{1}{\left[|X(t)| +|\dot{X}(t)|\right]}\right|^{\theta}\right).
\end{align*}
\end{proof}

According to Lemmas \ref{lip-theta} and \ref{rep} the sample paths of $G$ are $C^2([0,1],\mathbb{R})$ and $\|G\|_{\infty,2}$ belongs to $L^p$ for all $p >0$. On the other hand, the variance of the centered Gaussian random variable $G(s)$ is $\frac{\pi^2}{3}$, then
\begin{align*}
\sup_{s \in [0,1]} {\mathbb E} \left( \frac{1}{|\dot{G}(s)|^{\theta}}\right) <\infty,
\end{align*}
and $G$ fulfills the Hypothesis of Lemma \ref{lem-I5-X}, therefore we have the following result.
%\begin{lemma}\label{lem-I5}
%Let $0<\theta <1$, then there exists a positive constant $C$ such that
 %\begin{align*}
%\sup_{F \in {\mathcal L}_{1,\infty}({\mathbb R})} {\mathbb E} (F\circ N(G)) - {\mathbb E}\left( F \circ \Phi_{\delta} (G)\right) \leq C\delta^{\theta}.
%\end{align*}
%\end{lemma}
\begin{lemma}\label{lem-I5-regular}
Let $0<\theta<1$. Then, there exists a positive constant $C_{\theta}$ such that
\begin{align*}
\sup_{ F \in {\mathcal L}_1({\mathbb R})}{\mathbb E} \left(F\circ N(G)\right) 
- {\mathbb E} \left(F\circ \Phi_{\delta}(G) \right) \leq C_{\theta}\delta^{\theta}.
\end{align*}
\end{lemma}

\begin{proof}
Since $F\in{\mathcal L}_1(\mathbb{R})$ we have
\begin{align*}
{\mathbb E} \left(F\circ N(G)\right) 
- {\mathbb E} \left(F\circ \Phi_{\delta}(G) \right)&\leq \mathbb{E}\left(\left|N(G)-\Phi_{\delta}(G)\right|\right)\\
&\leq\mathbb{E}\left(\left[N(G)+\Phi_{\delta}(G)\right]{\mathbf 1}_{{\mathcal A}_{G}\leq \delta}\right),
\end{align*}
see notation \eqref{notation_A} for the definition of the random variable ${\mathcal A}_{G}$.  
Let $0<\theta<\tilde{\theta}<1$ and $q\geq 1$ such that $\frac{\tilde{\theta}}{q}=\theta$ and it is verified that $\frac{1}{p}+\frac{1}{q}=1$. Then, the H\"older and Jensen inequalities imply that
\begin{align*}
{\mathbb E} \left(F\circ N(G)\right) 
- {\mathbb E} \left(F\circ \Phi_{\delta}(G) \right)\leq 2^{1-\frac{1}{p}}\mathbb{E}\left(|N(G)|^p+|\Phi_{\delta}(G)|^p\right)^{\frac{1}{p}}\left[\mathbb{P}\left({\mathcal A}_{G}\leq \delta\right)\right]^{\frac{1}{q}}.
\end{align*}
Since the process $G$ satisfies hypothesis \eqref{hipo}, the inequalities \eqref{bound_A} and \eqref{bound_A_2} imply that 
\begin{equation*}
\mathbb{P}\left({\mathcal A}_{G}\leq \delta\right)\leq C_{\tilde{\theta}}\delta^{\tilde{\theta}},
\end{equation*}
therefore
\begin{align*}
{\mathbb E} \left(F\circ N(G)\right) 
- {\mathbb E} \left(F\circ \Phi_{\delta}(G) \right)\leq C_{\theta}\mathbb{E}\left(|N(G)|^p+|\Phi_{\delta}(G)|^p\right)^{\frac{1}{p}}\delta^{\theta}.
\end{align*}
Recall that
\begin{align*}
G(t)= \int_0^1 \cos (\pi tu) dB^1(u)+ \int_0^1 \sin (\pi tu) dB^2(u).
\end{align*}
Using \cite[Proposition (2.73)]{jacod}, we can derive  with respect to the variable $t$  under the stochastic integrals, therefore
\begin{align*}
\dot{G}(t)= \pi \left[\int_0^1 u \cos (\pi tu) dB^1(u)+ \int_0^1 u\sin (\pi tu) dB^2(u)\right],
\end{align*}
and $$Var(\dot{G}(t))=\pi^2 \int_0^1 u^2 du=\frac{2\pi^2 }{3}>0.$$ According to Corollary 3.7 in \cite{Wse-Aza}, it is satisfied that
\begin{align}\label{moment-nimber-dotG}
{\mathbb E} ( N(\dot{G})^p )<\infty,~~\forall p \geq 1.
\end{align}

Now using that $\Phi_{\delta}(G)\leq N(\dot{G})+1$ (see proof of Proposition 13 in \cite{JMA-Leon}) and Corollary $3.7$ in \cite{Wse-Aza} we get
\begin{align*}
{\mathbb E} \left(F\circ N(G)\right) 
- {\mathbb E} \left(F\circ \Phi_{\delta}(G) \right)&\leq C_{\theta}\left[\mathbb{E}\left(|N(G)|^p\right)+2^{p-1}\mathbb{E}\left(|N(\dot{G})|^p\right)+2^{p-1}\right]^{\frac{1}{p}}\delta^{\theta}\\
&\leq C_{\theta}\delta^{\theta}.
\end{align*}
\end{proof}

\subsection{Estimation of $I_2$ in the regular case}

If the coefficients of the polynomials $\{X_m\}_m$ in \eqref{pol} satisfy Hypothesis \ref{H2} we have the following lemma. 

\begin{lemma}\label{bound}
Consider the polynomials $\{X_m\}_m$ defined in \eqref{pol} and define $V_m(t)=(X_m(t),\dot{X}_m(t))$. If the density of $\{(x_r,y_r), r=0, 1, \ldots\}$ satisfies the Hypothesis \ref{H2} then
\begin{equation*}
\sup_{t\in[0,1]}\sup_{m}p_{V_m,t}(x,y)\leq \frac{C}{Q(x,y)},
\end{equation*}
where $Q\;:\mathbb{R}^2\to \mathbb{R}$ is a polynomial, $C$ is a positive constant and $p_{V_{m,t}}$ is the density function of $V_m$.
\end{lemma}

\begin{proof}
The covariance matrix of the random vector $V_m$ is
\begin{align*}
\Gamma_m = \left( \begin{array} {cc}
1 & 0\\
0 & \frac{\pi^2(2m-1)(m-1)}{6m^2} \end{array}
\right).
\end{align*}

In \cite[Lemma 7]{JMA-Leon} the authors provide useful bounds for the density function $p_{V_m,t}$ in the case where $m$ is large enough and $t$ does not belong to the following set, $D_n^{\epsilon}=\{(t_1,\ldots , t_n)\in [0,1]^n\;:\; \exists i\neq j \text{ such that } |t_i-t_j|<\epsilon\},$ for $\epsilon >0$. Note that in our case $n=1$ and for $m>1$, $\det(\Gamma_m)\neq 0$, therefore the matrix $\Gamma_m$ does not depend on the parameter $t$.

On the other hand, $\Gamma_m$ converges uniformly over $[0,1]$ to the covariance matrix, $\Gamma_t$, of the random vector $(G(t),\dot{G}(t))$ (see equality \eqref{cm_g}) when $m\to \infty$. In addition, since $\det(\Gamma_t)\neq 0$ and $\det(\Gamma_m)\neq 0$ for all $t\in [0,1]$ the result in \cite[Lemma 12]{JMA-Leon} is satisfied, but it is not necessary to restrict the parameter $t$ to the set $D_n^{\epsilon}$ with $\epsilon>0$ and therefore $m$ does not depend on $\epsilon$. Finally, we can replicate the proof of \cite[Lemma 7]{JMA-Leon} and to get the result. 
\end{proof}

\begin{lemma}\label{lemma-I2-R}
Assume that the  coefficients of the polynomials $\{X_m\}_m$ in \eqref{pol} satisfy the Hypothesis \ref{H2}.
Then there exists a positive constant $C$ satisfying that
\begin{align*}
\sup_{ F \in {\mathcal L}_1(C^1([0,1],{\mathbb R}))}{\mathbb E} \left(F\circ \Phi_{\delta,\varepsilon}(X_m) \right) 
- {\mathbb E} \left(F\circ \Phi_{\delta}(X_m) \right) \leq C \frac{\varepsilon}{\delta}.
\end{align*}
\end{lemma}
\begin{proof}
Let $F$ be in ${\mathcal L}_1(C^1(0,1),{\mathbb R}))$ then
\begin{align*}
\delta{\mathbb E} \left(F\circ\Phi_{\delta,\varepsilon}(X_m) \right) 
- {\mathbb E} \left(F\circ \Phi_{\delta}(X_m) \right) &\leq \delta {\mathbb E}\left| \Phi_{\delta,\varepsilon}(X_m) -\Phi_{\delta}(X_m) \right|\\
&\leq \frac{1}{2}\int_0^1 {\mathbb E} \left( |\dot{X_m}(s) | \frac{||X_m(s)| - \delta -\varepsilon|}{\varepsilon} {\mathbf 1}_{\delta \leq |X_m(s)| \leq \delta +\varepsilon} \right)ds.
\end{align*}

According to Lemma  \ref{bound} for $Q(x,y)=(1+y^4)$ there exists a constant $C$ such that $p_{V_{m,t}}$ the density of the random vector $(X_m(t),\dot{X_m}(t))$ is bounded by
\begin{align*}
\sup_{t\in [0,1]}p_{V_{m,t}}(x,y)\leq  \frac{C}{1+y^4}
\end{align*}
and
\begin{align*}
\delta{\mathbb E} \left(F\circ\Phi_{\delta,\varepsilon}(X_m) \right) 
- {\mathbb E} \left(F\circ \Phi_{\delta}(X_m) \right) 
&\leq \frac{C}{\varepsilon}\int_0^1\int_{\mathbb{R}}||x| - \delta -\varepsilon| {\mathbf 1}_{\delta \leq |x| \leq \delta +\varepsilon}\int_0^{\infty}\frac{|y|}{1+y^4}dydxds\\
&\leq C \varepsilon.
\end{align*}
\end{proof}

\subsection{Estimation of $I_2$}

\begin{lemma}\label{lemma-I2}
For all $\{(x_r,y_r), r=0, 1, \ldots\}$ sequence of i.i.d. random vectors with zero mean, identity covariance matrix and $4-$integrable there exists a positive constant $C$ such that

\begin{align*}
\sup_{ F \in {\mathcal L}_{1\infty}({\mathbb R})}{\mathbb E} \left(F\circ \Phi_{\delta,\varepsilon}(X_m) \right) 
- {\mathbb E} \left(F\circ \Phi_{\delta}(X_m) \right) \leq C {\mathbb E}(|x_0|^4+|y_0|^4)\left[  \frac{\varepsilon}{\delta} + \frac{\ln m }{  m^{1/6}}\frac{(\varepsilon+1)}{\delta \varepsilon}\right].
\end{align*}
\end{lemma}

\begin{proof}
Let $F$ be a Lipschitz continuous function bounded by $1$, then
\begin{align*}
{\mathbb E} \left(F\circ \Phi_{\delta,\varepsilon}(X_m) \right) 
- {\mathbb E} \left(F\circ \Phi_{\delta}(X_m) \right) \leq {\mathbb E} \left(\left| \Phi_{\delta,\varepsilon}(X_m) -\Phi_{\delta}(X_m) \right|\right).
\end{align*}
Note that for any function $f$ in $C^1([0,1],\mathbb{R})$
\begin{align*}
0\leq \Phi_{\delta,\varepsilon}(f) -\Phi_{\delta}(f) &\leq \frac{1}{\delta} \int_0^1 |\dot{f}(s)| \frac{|\delta +\varepsilon -|f(s)||}{\varepsilon} {\mathbf 1}_{\delta \leq |f(s)| \leq \delta +\varepsilon} ds \\
& \leq \tilde{\psi}_{\delta,\varepsilon}(f),
\end{align*}
where $\tilde{\psi}_{\delta,\varepsilon}$ is defined on $C^{1}( [0,1],{\mathbb R})$ as
\begin{align*}
\tilde{\psi}_{\delta,\varepsilon}(f) =\frac{1}{\delta} \int_0^1 |\dot{f}(s)| \tilde{H}_{\delta,\varepsilon}(|f(s)|) ds
\end{align*}
and $\varepsilon\tilde{H}_{\delta,\varepsilon}$ is a $1-$Lipschitz continuous function on ${\mathbb R}$ given by 
\begin{align*}
\tilde{H}_{\delta,\varepsilon}(x)= 
\begin{cases}
0 & \mbox{if } x >\delta + \varepsilon, \\
\frac{ \delta +  \varepsilon-x }{\varepsilon} &\mbox{if }  \delta<x <\delta + \varepsilon, \\
\frac{\delta-x}{\varepsilon} &\mbox{if } \delta - \varepsilon <x <\delta, \\
0 & \mbox{if }x <-\delta -\varepsilon.
\end{cases}
\end{align*}
Then,
\begin{align*}
{\mathbb E} \left(F\circ \Phi_{\delta,\varepsilon}(X_m) \right) 
- {\mathbb E} \left(F\circ \Phi_{\delta}(X_m) \right) \leq {\mathbb E} \left(\tilde{\psi}_{\delta,\varepsilon}(X_m)\right).
\end{align*}
Similarly to \eqref{b-3} we can conclude that $\frac{C\delta \varepsilon}{\varepsilon+1}\tilde{\psi}_{\delta,\varepsilon}$ belongs to ${\mathcal Lip}_{1,1}(C^1([0,1],{\mathbb R}))$
and
\begin{align*}
{\mathbb E} \left(F\circ \Phi_{\delta,\varepsilon}(X_m) \right) 
- {\mathbb E} \left(F\circ \Phi_{\delta}(X_m) \right) \leq {\mathbb E} \left(\tilde{\psi}_{\delta,\varepsilon}(X_m)\right) - {\mathbb E} \left(\tilde{\psi}_{\delta,\varepsilon}(G)\right)+ {\mathbb E} \left(\tilde{\psi}_{\delta,\varepsilon}(G)\right).
\end{align*}
According to  Proposition \ref{Donsker-loc} 
\begin{align*}
{\mathbb E} \left(F\circ \Phi_{\delta,\varepsilon}(X_m) \right) 
- {\mathbb E} \left(F\circ \Phi_{\delta}(X_m) \right) \leq C {\mathbb E}(|x_0|^4+|y_0|^4)\frac{(\varepsilon+1)}{\delta \varepsilon}\frac{\ln m }{ m^{1/6}}+ {\mathbb E} \left(\tilde{\psi}_{\delta,\varepsilon}(G)\right).
\end{align*}
Using the same arguments as in the proof of Lemma \ref{lemma-I4} we obtain
that
\begin{align*}
{\mathbb E} \left(\tilde{\psi}_{\delta,\varepsilon}(G)\right)\leq C\frac{\varepsilon}{\delta}
\end{align*}
and 
\begin{align*}
{\mathbb E} \left(F\circ \Phi_{\delta,\varepsilon}(X_m) \right) 
- {\mathbb E} \left(F\circ \Phi_{\delta}(X_m) \right) \leq C{\mathbb E}(|x_0|^4+|y_0|^4)\left[ \frac{(\varepsilon+1)}{\delta \varepsilon}\frac{\ln m }{ m^{1/6}}+ C\frac{\varepsilon}{\delta}\right].
\end{align*}
\end{proof}

\subsection{Estimation of $I_1$}
\begin{lemma}
\label{lemma-I1}
Let $0<\theta <1.$ There exists a positive constant $C_{\theta}$ such that for all $\{(x_r,y_r), r=0, 1, \ldots\}$ sequence of i.i.d. random vectors with zero mean, identity covariance matrix and $4-$integrable it is satisfied that
\begin{align*}
\sup_{ F \in {\mathcal L}_{1,\infty}(\mathbb R)}{\mathbb E} \left(F(N(X_m)) \right) 
- {\mathbb E} \left(F\circ \Phi_{\delta}(X_m) \right) \leq C_{\theta} {\mathbb E}(|x_0|^4+|y_0|^4)\left[  (\varepsilon+\delta)^{\theta} + \frac{1}{\varepsilon}\frac{\ln m }{  m^{1/6}}\right].
\end{align*}
\end{lemma}
\begin{proof}
Here we use the notation defined in \eqref{notation_A}.  According to Lemma \ref{lem-n-delta-sta},
on the set 
$$\left\{\omega, \delta < \mathcal{A}_{X_m}\right\}$$
$N(X_m) =\Phi_{\delta}(X_m).$
Let $F\in {\mathcal L}_{\infty,1}({\mathbb R}).$ By definition $F$ is bounded by $1$ and  
\begin{align*}
\left| F(N(X_m)) - F\circ \Phi_{\delta}(X_m)\right|\leq 2 {\mathbf 1}_{\mathcal{A}_{X_m}\leq \delta}.
\end{align*}
Then
\begin{align*}
{\mathbb E} \left(F(N(X_m,[0,1])) \right) 
- {\mathbb E} \left(F\circ \Phi_{\delta}(X_m) \right) \leq 2{\mathbb P}( \mathcal{A}_{X_m}\leq \delta).
\end{align*}
Let $\overline{H}_{\delta,\varepsilon}$ be a real $1/\varepsilon-$Lipschitz continuous function, bounded by $1$, defined as follows
\begin{align*}
\overline{H}_{\delta,\varepsilon}(x)= 
\begin{cases}
1, & \text{if } x\leq \delta\\
\frac{\delta+\varepsilon -x}{\varepsilon},  & \text{if } \delta\leq x\leq \delta+\varepsilon\\
0, & \text{else,} 
\end{cases}
\end{align*}
and $$\overline{\Psi}_{\delta,\varepsilon}(f)= \overline{H}_{\delta,\varepsilon}( \mathcal{A}_f).$$
Note that the function $\overline{\Psi}_{\delta,\varepsilon}$ is a $1/\varepsilon-$Lipschitz continuous function bounded by $1$ and 

\begin{align*}
{\mathbf 1}_{\mathcal{A}_{X_m}\leq \delta}\leq \overline{\Psi}_{\delta,\varepsilon}(X_m).
\end{align*}

Then, we estimate 
\begin{align*}
\frac{1}{2}\left[{\mathbb E} \left(F(N(X_m)) \right) 
- {\mathbb E} \left(F\circ \Phi_{\delta}(X_m) \right) \right]\leq \left[ {\mathbb E}(\overline{\Psi}_{\delta,\varepsilon}(X_m)) -{\mathbb E}(\overline{\Psi}_{\delta,\varepsilon}(G))\right] + {\mathbb E}(\overline{\Psi}_{\delta,\varepsilon}(G)).
\end{align*}
Using Proposition \ref{Donsker-loc} we get
\begin{align*}
\frac{1}{2}\left[{\mathbb E} \left(F(N(X_m,[0,1])) \right) 
- {\mathbb E} \left(F\circ \Phi_{\delta}(X_m) \right) \right]\leq
 C{\mathbb E}\left( |x_0|^4+|y_0|^4 \right)\frac{1}{\varepsilon}
\frac{\ln m}{m^{\frac{1}{6}} } +  {\mathbb E}(\overline{\Psi}_{\delta,\varepsilon}(G)).
\end{align*}
Now, note that
\begin{align*}
 {\mathbb E}(\overline{\Psi}_{\delta,\varepsilon}(G))\leq {\mathbb P}\left(\mathcal{A}_{G}\leq \delta+\varepsilon\right).
\end{align*}
Using the same computations as in the proof of Lemma \ref{lem-I5-X} (see inequalities \eqref{bound_A} and \eqref{bound_A_2}) we obtain
that
\begin{align*}
{\mathbb E} \left(\overline{\Psi}_{\delta,\varepsilon}(G)\right)\leq C_{\theta}(\varepsilon +\delta)^{\theta}
\end{align*}
and therefore 
\begin{align*}
{\mathbb E} \left(F(N(X_m,[0,1])) \right) 
- {\mathbb E} \left(F\circ \Phi_{\delta}(X_m) \right)\leq
 C_{\theta}
\left[{\mathbb E}\left( |x_0|^4+|y_0|^4 \right)\frac{\ln m}{m^{\frac{1}{6}} } \frac{1}{\varepsilon}+ (\varepsilon +\delta)^{\theta}\right].
\end{align*}
\end{proof}

\subsection{Estimation of $I_1$ in the regular case}
\begin{lemma}\label{maj-norm-sup}
Assume that the coefficients of the polynomials $\{X_m\}_m$ in \eqref{pol} satisfy the Hypothesis \ref{H2}. Then for all $p\in \mathbb{N}$ it is satisfied that 
\begin{align*}
\sup_{m}\mathbb{E}(N(X_m)^p)<\infty.
\end{align*}
\end{lemma}

\begin{proof}
According Proposition 6 in \cite{JMA-Leon} it is satisfied that
\begin{align*}
\lim_{m\to \infty} \mathbb{E}(N(X_m)\left[N(X_m)-1\right]\cdots\left[N(X_m)-p+1\right])=\mathbb{E}(N(G)\left[N(G)-1\right]\cdots\left[N(G)-p+1\right]),
\end{align*}
therefore 
\begin{align*}
\lim_{m\to \infty} \mathbb{E}\left(N(X_m)^p\right)=\mathbb{E}(N(G)^p)
\end{align*}
and in consequence the sequence $\{N(X_m)^p\}_{m\in\mathbb{N}}$ is bounded which implies the result.
\end{proof}

\begin{lemma}\label{lem-I1-X-R}
Assume that the coefficients of the polynomials $\{X_m\}_m$ in \eqref{pol} satisfy the Hypothesis \ref{H2}. Let $0<\theta<1$, then there exists a positive constant $C_{\theta}$ such that it is satisfied that
\begin{align*}
\sup_{ F \in {\mathcal L}_1({\mathbb R})}{\mathbb E} \left(F\circ N(X_m)\right) 
- {\mathbb E} \left(F\circ \Phi_{\delta}(X_m) \right) \leq C\mathbb{E}(|x_0|^4+|y_0|^4)\frac{\ln m}{m^{\frac{1}{6}}}\frac{1}{\delta}\left(1+\frac{1}{\varepsilon}\right)+C_{\theta}\delta^{\theta} .
\end{align*}
\end{lemma}

\begin{proof}
Since $F \in {\mathcal L}_1({\mathbb R})$ we have
\begin{align*}
{\mathbb E} \left(F\circ N(X_m)\right) 
- {\mathbb E} \left(F\circ \Phi_{\delta}(X_m) \right)&\leq \mathbb{E}\left(|N(X_m)-\Phi_{\delta}(X_m)|\right)\\
&\leq \mathbb{E}\left(N(X_m){\mathbf 1}_{{\mathcal A}_{X_m}\leq \delta}\right)+\mathbb{E}\left(\Phi_{\delta}(X_m){\mathbf 1}_{{\mathcal A}_{X_m}\leq \delta}\right),\\
\end{align*}
where the random variable ${\mathcal A}_{X_m}$ is defined as in \eqref{notation_A}. 

Let $0<\theta<\tilde{\theta}<1$ and $q\geq 1$ such that $\frac{\tilde{\theta}}{q}=\theta$ and $p$ is the conjugate exponent of $q$. Then applying H\"older inequality 
\begin{align*}
\mathbb{E}\left(N(X_m){\mathbf 1}_{{\mathcal A}_{X_m}\leq \delta}\right)\leq \mathbb{E}\left(N(X_m)^p\right)^{\frac{1}{p}}\left[\mathbb{P}\left({\mathcal A}_{X_m}\leq \delta\right)\right]^{\frac{1}{q}}.
\end{align*}
The inequalities \eqref{bound_A} and \eqref{bound_A_2} imply that 
\begin{equation}\label{bProb}
\mathbb{P}\left({\mathcal A}_{X_m}\leq \delta\right)\leq C_{\tilde{\theta}}\delta^{\tilde{\theta}},
\end{equation}
therefore, applying inequality \eqref{bProb} and Lemma \ref{maj-norm-sup} we get
\begin{align}\label{ter-1}
\mathbb{E}\left(N(X_m){\mathbf 1}_{{\mathcal A}_{X_m}\leq \delta}\right)\leq C_{\theta}\delta^{\theta}.
\end{align}
On the other hand, note that
\begin{align*}
\mathbb{E}\left(\Phi_{\delta}(X_m){\mathbf 1}_{{\mathcal A}_{X_m}\leq \delta}\right)\leq &\left[\mathbb{E}\left(\Phi_{\delta,\varepsilon}(X_m)\overline{\Psi}_{\delta,\varepsilon}(X_m)\right)-\mathbb{E}\left(\Phi_{\delta,\varepsilon}(G)\overline{\Psi}_{\delta,\varepsilon}(G)\right)\right]\\
&+\mathbb{E}\left(\Phi_{\delta,\varepsilon}(G)\overline{\Psi}_{\delta,\varepsilon}(G)\right).
\end{align*}
Let $x,y \in C^1([0,1],{\mathbb R}).$ The function $\overline{\Psi}_{\delta,\varepsilon}$ is $1/\varepsilon-$Lipschitz continuous bounded by $1$. Then
\begin{align*}
\left|\Phi_{\delta,\varepsilon}(x)\overline{\Psi}_{\delta,\varepsilon}(x)-\Phi_{\delta,\varepsilon}(y)\overline{\Psi}_{\delta,\varepsilon}(y)\right|&\leq \left|\Phi_{\delta,\varepsilon}(x)-\Phi_{\delta,\varepsilon}(y)\right|+|\Phi_{\delta,\varepsilon}(y)|\frac{1}{\varepsilon}\|x-y\|_{\infty}\\
&\leq  \left|\Phi_{\delta,\varepsilon}(x)-\Phi_{\delta,\varepsilon}(y)\right|+\frac{1}{2\delta \varepsilon}\|\dot{y}\|_{\infty}\|x-y\|_{\infty}.
\end{align*}
Using inequality \eqref{b-3} we get 
\begin{align*}
\left|\Phi_{\delta,\varepsilon}(x)\overline{\Psi}_{\delta,\varepsilon}(x)-\Phi_{\delta,\varepsilon}(y)\overline{\Psi}_{\delta,\varepsilon}(y)\right|\leq \frac{1}{\delta}\|\dot{x}-\dot{y}\|_{\infty}+\frac{1}{\delta \varepsilon}\|\dot{y}\|_{\infty}\|x-y\|_{\infty},
\end{align*}
then
\begin{align*}
\frac{C\delta\varepsilon}{\varepsilon+1}\left|\Phi_{\delta,\varepsilon}(x)\overline{\Psi}_{\delta,\varepsilon}(x)-\Phi_{\delta,\varepsilon}(y)\overline{\Psi}_{\delta,\varepsilon}(y)\right|\leq \left[\|y\|_{\infty,1}+\|x-y\|_{\infty,1}\right]\|x-y\|_{\infty,1}.
\end{align*}
In consequence $\frac{C\delta \varepsilon}{1+\varepsilon} \Phi_{\delta,\varepsilon}\in{\mathcal Lip}_{1,1}(C^1([0,1],{\mathbb R}))$, thus Proposition \ref{Donsker-loc} yields 
\begin{align*}
\mathbb{E}\left(\Phi_{\delta}(X_m){\mathbf 1}_{{\mathcal A}_{X_m}\leq \delta}\right)\leq \frac{C(\varepsilon+1)}{\delta \varepsilon}\mathbb{E}(|x_0|^4+|y_0|^4)\frac{\ln m}{m^{\frac{1}{6}}}+\mathbb{E}\left(\Phi_{\delta,\varepsilon}(G)\overline{\Psi}_{\delta,\varepsilon}(G)\right).
\end{align*}
Finally, 
\begin{align*}
\mathbb{E}\left(\Phi_{\delta,\varepsilon}(G)\overline{\Psi}_{\delta,\varepsilon}(G)\right)\leq \mathbb{E}\left(\Phi_{\delta+\varepsilon}(G){\mathbf 1}_{{\mathcal A}_{G}\leq \delta+\varepsilon}\right),
\end{align*}
and using again that $\Phi_{\delta+\varepsilon}(G)\leq N(\dot{G})+1$ (see proof of Proposition 13 in \cite{JMA-Leon}),
the fact that $N(\dot{G})$ has finite moment of any order (see \eqref{moment-nimber-dotG}) and H\"older inequality it is obtained that
\begin{align*}
\mathbb{E}\left(\Phi_{\delta,\varepsilon}(G)\overline{\Psi}_{\delta,\varepsilon}(G)\right)\leq 2^{1-\frac{1}{p}}\left[\mathbb{E}(|N(\dot{G})|^p)+1\right]^{\frac{1}{p}}\left[\mathbb{P}\left({\mathcal A}_{G}\leq \delta+\varepsilon\right)\right]^{\frac{1}{q}},
\end{align*}
thus, taking $\varepsilon<\delta$
\begin{align*}
\mathbb{E}\left(\Phi_{\delta,\varepsilon}(G)\overline{\Psi}_{\delta,\varepsilon}(G)\right)\leq C_{\theta}\delta^{\theta}.
\end{align*}
\end{proof}

%%%%%%%%%%%%%%%%%%%%%%%%%%%%%%%%%%%%%%%%%%%%%%%%%
%%%%%%%%%%%%%%%%%%%%%%%%%%%%%%%%%%%%%%%%%%%%%%%%%
%%%%%%%%%%%%%%%%%%%%%%%%%%%%%%%%%%%%%%%%%%%%%%%

\section{Proof of Theorem \ref{Z_Rate}}\label{NR-case}
Let $F\in {\mathcal L}_{1,\infty}({\mathbb R}).$
Let us recall decomposition defined in \eqref{sumII}
\begin{align*}
&{\mathbb E}\left(F\left(N(X_m,[0,1])\right)\right)-{\mathbb E}\left(F\left(N(G,[0,1])\right)\right)= \sum_{i=1}^5I_i.
\end{align*}

The terms $I_i$ for $i=1,...,5$ are bounded in Lemmas \ref{lemma-I3}, \ref{lemma-I4},  \ref{lem-I5-regular}, \ref{lemma-I2} and \ref{lemma-I1}. Therefore, for all $0<\theta<1$ there exists a 
positive constant $C_{\theta}$ depending only on $\theta$ such that 
\begin{align*}
{\mathbb E}\left(F\left(N(X_m,[0,1])\right)\right)-{\mathbb E}\left(F\left(N(G,[0,1])\right)\right)&\leq C_{\theta} \left[ \frac{1}{\delta \varepsilon}
{\mathbb E}\left( |x_0|^4 +|y_0|^4\right) \frac{\ln m}{m^{1/6}}+\frac{\varepsilon}{\delta} + \varepsilon^{\theta}+\delta^{\theta}\right]\\
&\leq C_{\theta} {\mathbb E}\left( |x_0|^4 +|y_0|^4\right) \left[ \frac{1}{\delta \varepsilon}
\frac{\ln m}{m^{1/6}}+\frac{\varepsilon}{\delta} + \varepsilon^{\theta}+\delta^{\theta}\right]
\end{align*}

Let $ f_c(\delta,\varepsilon)= \frac{c}{\varepsilon \delta} +\frac{\varepsilon}{\delta} +\varepsilon^{\theta} +\delta^{\theta}$.
Taking $c=\frac{\ln m}{m^{1/6}}$, $\varepsilon=\sqrt{c}$ and $\delta=c^{\frac{1}{2(\theta+1)}}$ we get

\begin{align*}
\frac{c}{\varepsilon \delta} +\frac{\varepsilon}{\delta} +\varepsilon^{\theta} +\delta^{\theta}=3c^{\frac{\theta}{2(\theta+1)}}+c^{\frac{\theta}{2}}=O\left(\frac{\ln m}{m^{\frac{\theta}{12(\theta+1)}}}\right).
\end{align*}
Note that $\sup_{\theta \in ]0,1[} \frac{\theta}{2(\theta +1)}=\frac{1}{4}.$

Let $\alpha <\frac{1}{24},$ there exists $\theta <1$ such that $\alpha < \frac{\theta}{2(\theta +1)}$ and
there exists a constant $c_{\alpha}$ such that
\begin{align*}
C_{\theta}
\left( \frac{\ln m}{m^{\frac{1}{6}}}\right)^{\frac{\theta}{2(\theta +1)}} \leq c_{\alpha} m^{-\alpha}.
\end{align*}

$\Box$

%%%%%%%%%%%%%%%%%%%%%%%%%%%%%%%%%%%%%%%%%%%%%%%%%%%%%%%%%

\section{Proof of Theorem \ref{Z_Rate_regular}}\label{R-case}
Let $F\in {\mathcal L}_{1}({\mathbb R}).$
The terms $I_i$ for $i=1,...,5$ are bounded in Lemmas \ref{lemma-I3}, \ref{lemma-I4}, \ref{lem-I5-regular}, \ref{lemma-I2-R} and \ref{lem-I1-X-R}. Therefore, for all $0<\theta<1$ there exists a 
constant $C_{\theta}$ depending only on $\theta$ such that 
\begin{align*}
{\mathbb E}\left(F\left(N(X_m,[0,1])\right)\right)-{\mathbb E}\left(F\left(N(G,[0,1])\right)\right)&\leq C_{\theta} \left[ \frac{1}{\delta \varepsilon}
{\mathbb E}\left( |x_0|^4 +|y_0|^4\right) \frac{\ln m}{m^{1/6}}+\frac{\varepsilon}{\delta}+\delta^{\theta}\right]\\
&\leq C_{\theta} {\mathbb E}\left( |x_0|^4 +|y_0|^4\right) \left[ \frac{1}{\delta \varepsilon}
\frac{\ln m}{m^{1/6}}+\frac{\varepsilon}{\delta}+\delta^{\theta}\right]
\end{align*}

Let $ \tilde{f}_c(\delta,\varepsilon)= \frac{c}{\varepsilon \delta} +\frac{\varepsilon}{\delta}+\delta^{\theta}$. Its critical point $(\delta_c,\varepsilon_c) $  is solution of the system
\begin{align*}
-\frac{c}{\delta^2 \varepsilon}-\frac{\varepsilon}{\delta^2} +\theta \delta^{\theta -1}=0\\
-\frac{c}{\delta \varepsilon^2} + \frac{1}{\delta}=0.
\end{align*}
Then $\varepsilon= \sqrt{c}$ and $\delta =\left( \frac{2  \sqrt{c}}{\theta}\right)^{\frac{1}{\theta +1}}$ therefore
\begin{align*}
\tilde{f}_c( \delta_c,\varepsilon_c)=
\left[ 2^{\frac{\theta}{\theta+1}}\theta^{\frac{1}{\theta+1}} +\left(\frac{2  }{\theta}\right)^{\frac{\theta}{\theta +1}} \right] 
 c^{\frac{\theta}{2(\theta +1)}} .
\end{align*}
Taking $c=\frac{\ln m}{m^{1/6}}$ in the above expression we get 
\begin{align*}
{\mathbb E}\left(F\left(N(X_m,[0,1])\right)\right)-{\mathbb E}\left(F\left(N(G,[0,1])\right)\right)&\leq C_{\theta}
\left( \frac{\ln m}{m^{\frac{1}{6}}}\right)^{\frac{\theta}{2(\theta +1)}}.
\end{align*}
Note that $\sup_{\theta \in ]0,1[} \frac{\theta}{2(\theta +1)}=\frac{1}{4}.$

Let $\alpha <\frac{1}{24},$ there exists $\theta <1$ such that $\alpha < \frac{\theta}{2(\theta +1)}$ and
there exists a constant $c_{\alpha}$ such that
\begin{align*}
C_{\theta}
\left( \frac{\ln m}{m^{\frac{1}{6}}}\right)^{\frac{\theta}{2(\theta +1)}} \leq c_{\alpha} m^{-\alpha}.
\end{align*}
$\Box$

%%%%%%%%%%%%%%%%%%%%%%%%%%%%%%%%%%%%%%%%%%%%%%%%%%%%%%%%%
%%%%%%%%%%%%%%%%%%%APPENDIX STEIN%%%%%%%%%%%%%%%%%%%%%%%%%%
%%%%%%%%%%%%%%%%%%%%%%%%%%%%%%%%%%%%%%%%%%%%%%%%%%%%%%
\appendix
\section{ Around Donsker Theorem}\label{stein+}

For the sake of completeness, this section is devoted to the proof of Theorem \ref{RC-B} and therefore Proposition \ref{Donsker-loc}. More precisely we prove the following.

\begin{proposition}\label{prop-donskerlip11}
 There exists a positive constant $c$ such that for all $\{X_r\}_{r\in\mathbb{N}}$ sequence of i.i.d. centered and $4-$integrable random vectors with identity covariance matrix it is satisfied that
 \begin{align*}
 \sup_{F \in {\mathcal Lip}_{1,1}(C([0,1],{\mathbb R}))}  {\mathbb E}( F(S^m)) -{\mathbb E}( F(B))\leq {c \|X\|_{L^4}^4 \frac{\ln {m}}{m^{\frac{1}{6}}}}.
 \end{align*}
\end{proposition} 
%%%%%%%%%%%%%%%%%%%%%%%%%%%%%%%%%%%%%%%%%%%%%%%%%%%%%%%%%%%%%%%%%%%%%%%%%%%%%%%%%%%%%%%%%%
%%%%%%%%%%%%%%%%%%Wiener space%%%%%%%%%%%%%%%%%
%%%%%%%%%%%%%%%%%%%%%%%%%%%%%%%%%%%%%%%%%%%%%%%%%%%%%%%%%%%%
\subsection{Wiener space}

Consider $Z=\{Z_n\}_{n\in {\mathbb N}}$ a sequence of independent standard Gaussian random variables and let $\{h_n\}_{\in {\mathbb N}}$ be an orthonormal basis of ${\mathbb H}$
where
\begin{align*}
{\mathbb H}=\left\{h\:|\: \exists !\;\dot{h} \in L^2([0,1],dt)~~\mbox{with}~~h(t)=\int_0^t\dot{h}(s) ds\right\}~~~\mbox{and}~~\|h\|_{{\mathbb H}}=\|\dot{h}\|_{L^2([0,1],dt)}.
\end{align*}
Then, we know from \cite{Nisio}, that when $N$ goes to infinity 

\begin{align}\label{ito-nisio}
\sum_{n=1}^N Z_n h_n \rightarrow B=\sum_{n=1}^{\infty} Z_nh_n ~~\mbox{~~in ~~} ~~C([0,1],{\mathbb R})~~\mbox{~~with probability~~}~~1.
\end{align}

By taking $d$ independent copies of the sequence $Z$, we construct the Wiener measure on $W=C([0,1],{\mathbb R}^d).$ We clearly have the diagram

\begin{align}\label{diag}
C([0,1],{\mathbb R}^d)^*\begin{array}{c}{\mathfrak e}^*\\
\rightarrow \end{array}{\mathbb H}^{\otimes d,*}\equiv {\mathbb H}^{\otimes d} =H\begin{array}{c}{\mathfrak e}\\
\rightarrow \end{array} C([0,1],{\mathbb R}^d)
\end{align}
where ${\mathfrak e }$ is the embedding from ${\mathbb H}^{\otimes d} $ into $C([0,1],{\mathbb R}^d).$ We note that $H$ is dense in $C([0,1],{\mathbb R}^d).$

Moreover, \eqref{ito-nisio} and the Parseval identity entails that for all $z\in W^*$
\begin{align}\label{car-funt}
{\mathbb E} \left(e^{i\langle z, B\rangle_{W,W^*} }\right)=\exp \left(-\frac{1}{2}\left\|{\mathfrak e}^*(z)\right\|_{H}\right).
\end{align}
We denote by $\mu$ the law of $B$ on $W=C([0,1],{\mathbb R}^d).$ Then the diagram \eqref{diag} and identity \eqref{car-funt} mean that $(H,W,{\mathfrak e})$ is a \textit{Wiener space}.

\begin{definition}(Wiener integral). The Wiener integral, denoted by $\delta$ is the isometric extension of the map
\begin{align*}
\delta ~:\begin{array}{c}
{\mathfrak e}(W^*)\subset{\mathbb H} \rightarrow L^2(\mu)\\
{\mathfrak e}^*(\eta)\mapsto \langle \eta,y\rangle_{W^*,W}
\end{array}
\end{align*}
\end{definition}

That means if $h=\lim_{n\rightarrow \infty} {\mathfrak e}^*(\eta_n)$ in $H$ then
\begin{align*}
\lim_{n\rightarrow \infty} \delta(\eta)_n=\delta(h),~~\mbox{~~in~~}~~L^2(\mu).
\end{align*}

\begin{definition}\label{cylindrical-funct}
Let $V$ be a Banach space. We say that a function $f: W \rightarrow V$ is cylindrical if it is of the form
\begin{align*}
f(y)=\sum_{j=1}^kf_j(\delta(h_1)(y),...\delta(h_k)(y))x_j,
\end{align*}
where for all $j=1,\ldots,k$, $f_j$ belongs to the Schwarz space on ${\mathbb R}^k$, $h_j$ is an element of $H$ and $x_j$ belongs to $V^k$. The space of such functions is denoted by ${\mathfrak C}(V).$
\end{definition}

The gradient of $f$ is then given by
\begin{align*}
\nabla f=\sum_{j,l=1}^k\partial_lf_j(\delta(h_1),...\delta(h_k))h_l\otimes x_j.
\end{align*}

The space ${\mathbb D}^{1,2}$ is the closure of ${\mathfrak C}(V)$ with respect to the norm of $L^2(W,\mu;H\otimes V).$
This construction can be iterated so that we can define higher order gradients (see \cite{ustunel} for details).
%%%%%%%%%%%%%%%%%%%%%%%%%%%%%%%%%%%%%%%%%%%%%%%%%%%%%%%%%%%%%%%%%%%%%%%%%%%%%%%%%%%%%%%%%
%%%%%%%%%%%%%%%%%%%%%Donsker theorem%%%%%%%%%%%%%%%%%%%%%%%%%%%%%%%%%%%%%%%%%%%%%%%%%%%%
%%%%%%%%%%%%%%%%%%%%%%%%%%%%%%%%%%%%%%%%%%%%%%%%%%%%%%%%%%%%%%%%%%%%%%%%%%%%%%%%%%%%%%%%

For $m\geq 1,$ let ${\mathcal D}^m=\left\{\frac{i}{m},~~i=0,...,m\right\}$ be the regular subdivision of $[0,1].$ Set ${\mathcal A}^m:=\{1,...,d\}\times\{0,...,m-1\}$
and for $(a_1,a_2) \in {\mathcal A}^m$ we define
\begin{align}\label{def-ham}
h_a^m(t):=\sqrt{m}\int_0^t{\mathbf 1}_{[\frac{a_2}{m},\frac{a_2+1}{m}]}(s) ds e_{a_1},
\end{align}
where $\{e_{a_1},~~a_1=1,...,d\}$ is the canonical basis of ${\mathbb R}^d.$ Consider
\begin{align*}
S^m=\sum_{a \in {\mathcal A}^m} X_ah_a^m
\end{align*}
where $\{X_{a}, a\in {\mathcal A}^m\}$ is a family of centered i.i.d. random vectors with identity of ${\mathbb R}^d$ covariance matrix. We denote by $X$ a random vector which have a common distribution in all its entries.

Remark that $\{h_a^m,~~a \in {\mathcal A}^m\}$ is an orthonormal family in $H.$
Let $\nu^m=\vect{h_a^m,~~a \in {\mathcal A}^m}\subset H$ and $\pi^m$ the orthogonal projection of $H$ on $\nu^m.$

Now, in view of the above, let $F \in{\mathcal Lip}_{1,1}(C([0,1],{\mathbb R}^d)),$ and we write
\begin{align}\notag
{\mathbb E}\left( F(S^m)\right) -{\mathbb E}\left( F(B)\right)=&{\mathbb E}\left( F(S^m)\right) -{\mathbb E}\left( F\circ\pi^N(S^m)\right) \\\notag
&+ {\mathbb E}\left( F\circ\pi^N(S^m)\right) -{\mathbb E}\left( F\circ\pi^N(B^m)\right)\\\notag
&+
{\mathbb E}\left( F\circ\pi^N(B^m)\right)-{\mathbb E}\left( F(B)\right)\\\label{sumofi}
=&\sum_{i=1}^3 I_i,
\end{align}
where $B^m$ is the linear interpolation of the Brownian motion $B$, that is to say
\begin{align*}
B^m=\sum_{a \in {\mathcal A}^m}\sqrt{m} \left( B^{a_1}\left(\frac{a_2+1}{m}\right) - B^{a_1}\left(\frac{a_2}{m}\right) \right)h_a^m(t).
\end{align*}

We recall that from Friz and Victoir \cite{friz} the following result is true.

\begin{theorem}
For all $p\geq 1$
\begin{align*}
{\mathbb E} \left( \|B\|_{C([0,1],{\mathbb R}^d)}^p\right)+ \sup_m {\mathbb E} \left( m^{p/2}\left\| B^m-B\right\|_{C([0,1],{\mathbb R}^d)}^p \right)<\infty.
\end{align*}
\end{theorem}

On the other hand, Theorem 3.2 in \cite{CouDec2020} also states that

\begin{theorem}\label{cont-proj}
There exists a positive constant $c$ such that if $\|X\|\in L^4$ then
\begin{align*}
{\mathbb E} \left( \|S^m\|_{C([0,1],{\mathbb R}^d)}^4\right)^{1/4}+\sup_{N,m}  N^{1/2}{\mathbb E} \left(\left\| S^m-\pi^N(S^m)\right)\|_{C([0,1],{\mathbb R}^d)}^4 \right)^{1/4}< c\|X\|_{L^4}.
\end{align*}
\end{theorem}

Using the above results we derive the following lemma where we estimate the terms $I_1$ and $I_3$ in equation \eqref{sumofi}.

\begin{lemma}\label{maj-I1-I3}
There exists a positive constant $c$ such that
if $X \in L^4$, then for all $F \in {\mathcal Lip}_{1,1}(C([0,1],{\mathbb R}^d))$
\begin{align*}
I_1+ I_3 \leq c \frac{1}{\sqrt{N}} \|X\|^4_{L^4},
\end{align*}
where $I_1$ and $I_3$ are defined in equality \eqref{sumofi}.
\end{lemma}
%%%%
\begin{proof}
We start with the term $I_1$. For all $F \in {\mathcal Lip}_{1,1}(C([0,1],{\mathbb R}^d))$ we have that
\begin{align*}
{\mathbb E}&\left(\left| F(S^m) -F(\pi^N(S^m))\right|\right) \\
&\leq  {\mathbb E}\left(\|S^m\|_{C([0,1],{\mathbb R}^d)} \|S^m-\pi^N(S^m)\|_{C([0,1],{\mathbb R}^d)} \right)\\
&+{\mathbb E}\left(\|S^m-\pi^N(S^m)\|^2_{C([0,1],{\mathbb R}^d)} \right),\\
\intertext{and using Cauchy-Schwarz inequality we get}
{\mathbb E}&\left(\left| F(S^m) -F(\pi^N(S^m))\right|\right) \\
&\leq \left[{\mathbb E}\left(\|S^m\|_{C([0,1],{\mathbb R}^d)}^2\right) {\mathbb E}\left(\|S^m-\pi^N(S^m)\|^2_{C([0,1],{\mathbb R}^d)} \right)\right]^{1/2}\\
&\quad+{\mathbb E}\left(\|S^m-\pi^N(S^m)\|^2_{C([0,1],{\mathbb R}^d)} \right).
\end{align*}

According to Theorem \ref{cont-proj} we derive that 
\begin{align*}
{\mathbb E}\left(\left| F(S^m) -F(\pi^N(S^m))\right|\right) \leq c\|X\|^4_{L^4}\frac{1}{\sqrt{N}}.
\end{align*}
The same computations hold for $I_3.$
\end{proof}

The main technical result of this section is the following adaptation of Theorem 3.3 of \cite{CouDec2020}, which allows us to bound the term $I_2$ in \eqref{sumofi}.

\begin{proposition}\label{prop-thme3.3}
There exists a positive constant $c$ such that if $X\in L^4$, then for any $F \in {\mathcal Lip}_{1,1}(C([0,1],{\mathbb R}^d))$ we have
\begin{align*}
{\mathbb E} \left( F\left(\pi^N(S^m)\right)\right)-{\mathbb E} \left( F\left(\pi^N(B^m)\right)\right)\leq c\|X\|_{L^4}^4\frac{\sqrt{N}}{\sqrt{m}}\ln\left( \frac{\sqrt{N}}{\sqrt{m}}\right).
\end{align*}
\end{proposition}

Combining Lemma \ref{maj-I1-I3}, Proposition \ref{prop-thme3.3} and identity \eqref{sumofi}, the global-upper bound of \eqref{sumofi} appears to be proportional to $1/\sqrt{N} + \sqrt{N}/\sqrt{m}\ln \left(\sqrt{N}/\sqrt{m}\right).$
See $N$ as a function of $m$ and note that this expression is minimal when $N$ is of the order of $m^{1/3}$ and, in consequence, we obtain Proposition \ref{prop-donskerlip11}.
%%%%%%%%%%%%%%%%%%%%%%%%%%%%%%%%%%%%%%%%%%%%%%%%%%%%%%%%%%%%%%%%%%%%%%%%%%%%
%%%%%%%%%%%%%%%%%%%%%Stein method%%%%%%%%%%%%%%%%%%%%%%%%%%%%%%%%%%%%%%%%%%%
%%%%%%%%%%%%%%%%%%%%%%%%%%%%%%%%%%%%%%%%%%%%%%%%%%%%%%%%%%%%%%%%%%%%%%%%%%%%%%%

\subsection{Stein method}
We will adapt the Section 4.1 of \cite{CouDec2020} to our setting. For the sake of simplicity, we denote $f_N =F\circ\pi^N.$

The Stein-Dirichlet formula (see \cite{decreu}) yields that,
\begin{align*}
{\mathbb E} \left( f_N(B^m) \right)-{\mathbb E} \left( f_N(S^m) \right)={\mathbb E} \left(\int_0^{\infty}
\frac{d}{du} P_{u}( f_N(S^m))du\right)
\end{align*}
where 
\begin{align*}
P_u(f_N)(x)={\mathbb E} \left( f_N\left(e^{-u} (x) + \beta_u B^m \right)\right),~~\forall x\in C([0,1],{\mathbb R}^d)
\end{align*}
with $\beta_u=\sqrt{1-e^{-2u}}.$

Moreover,  for any $\tau >0$
%t
\begin{align}\label{s-aux-ec}
{\mathbb E} &\left( f_N(B^m) \right)-{\mathbb E} \left( f_N(S^m) \right)\\\notag
&= {\mathbb E}\left( P_{\tau} \left(f_N\right)(S^m) \right)-{\mathbb E}\left(f_N(S^m) \right) +
{\mathbb E} \left(\int_{\tau}^{\infty}
L P_{u}( f_N(S^m))du\right),
\end{align}
where for $g~:\nu^m\rightarrow {\mathbb R}$ regular enough and for $x\in \nu^m \subset H,$

\begin{align*}
L(g)(x)= -\langle x,\nabla g(x) \rangle_H+\sum_{a \in {\mathcal A}^m} \left\langle \nabla^{(2)}g(x), h_a^m \otimes h_a^m \right\rangle_{H^{\otimes 2}}.
\end{align*}

For the first term in the right-hand side of equality \eqref{s-aux-ec} we have the following lemma.

\begin{lemma}\label{est-0-tau}
Assume that $X\in L^4$, then for any $F \in {\mathcal Lip}_{1,1}(C([0,1],{\mathbb R}^d))$ there exists a positive constant $c$ such that 
\begin{align*}
{\mathbb E}\left(f_N(S^m) \right) - {\mathbb E}\left( P_{\tau} \left(f_N\right)(S^m) \right)\leq c\|X\|_{L^4}^4\sqrt{1-e^{-2\tau}}.
\end{align*}
\end{lemma}

\begin{proof}
Since $F \in {\mathcal Lip}_{1,1}(C([0,1],{\mathbb R}^d))$ we have
\begin{align*}
&{\mathbb E}\left( P_{\tau} \left(f_N(S^m)\right) \right)-{\mathbb E}\left(f_N(S^m) \right)\\
&\leq 
{\mathbb E}\left(\|\pi^N(S^m)\|_{C([0,1],\mathbb{R}^d)}\left[|1-e^{-\tau}|\|\pi^N(S^m)\|_{C([0,1],\mathbb{R}^d)}+\beta_{\tau}\|\pi^N(B^m)\|_{C([0,1],\mathbb{R}^d)}\right]\right)\\
&\quad +{\mathbb E}\left(\|\pi^N(S^m)\|^2_{C([0,1],\mathbb{R}^d)}\right).
\intertext{Using Cauchy-Schwarz inequality and Theorem \ref{cont-proj} we obtain}
&{\mathbb E}\left( P_{\tau} \left(f_N(S^m)\right) \right)-{\mathbb E}\left(f_N(S^m) \right)\\
&\leq 
|1-e^{-\tau}|{\mathbb E}\left(\|\pi^N(S^m)\|^2_{C([0,1],\mathbb{R}^d)}\right)\\
&\quad+\beta_{\tau}{\mathbb E}\left(\|\pi^N(S^m)\|^2_{C([0,1],\mathbb{R}^d)}\right)^{1/2}{\mathbb E}\left(\|\pi^N(B^m)\|^2_{C([0,1],\mathbb{R}^d)}\right)^{1/2}\\
&\quad +{\mathbb E}\left(\|\pi^N(S^m)\|^2_{C([0,1],\mathbb{R}^d)}\right)\\
&\leq c\left((1-e^{-\tau})+\beta_{\tau}+1\right)\|X\|^4_{L^4}.
\end{align*}
Finally, applying the fact that $1-e^{-\tau}\leq 2\beta_{\tau}$
\begin{align*}
{\mathbb E}\left(f_N(S^m) \right) - {\mathbb E}\left( P_{\tau} \left(f_N\right)(S^m) \right)\leq c\|X\|_{L^4}^4\sqrt{1-e^{-2\tau}}.
\end{align*}

\end{proof}

Now, to estimate the second term in the right-hand side of equality \eqref{s-aux-ec} it is necessary to quote Lemma 4.3 of \cite{CouDec2020}.

\begin{lemma}
\label{lem4.3}
For any $F\in {\mathcal Lip}_{1,1}(C([0,1],{\mathbb R}^d))$ we have
\begin{align*}
&{\mathbb E}\left( LP_{\tau}(f_N(S^m)\right)\\
&= -\sum_{a \in {\mathcal A}^m} {\mathbb E}\left(
\left\langle\nabla^{(2)}(P_{\tau}(f_N)(S^m_{-a})-\nabla^{(2)}(P_{\tau}(f_N)(S^m),h_a^m\otimes h_a^m \right\rangle_{H^{\otimes 2}}\right)\\
&+\sum_{a \in {\mathcal A}^m} {\mathbb E}\left(X_a^2\int_0^1
\left\langle\nabla^{(2)}(P_{\tau}(f_N))(S^m_{-a}-rX_ah_a^m)-\nabla^{(2)}(P_{\tau}(f_N))(S^m),h_a^m\otimes h_a^m \right\rangle_{H^{\otimes 2}}dr
\right),
\end{align*} 
where $S^m_{-a}=S^m- X_a h_a^m.$
\end{lemma}

\begin{lemma}\label{lem-thme4.4}
There exists a positive constant $c$ such that for all $\tau >0,$ $F \in {\mathcal Lip}_{1,1}(C([0,1],{\mathbb R}^d)),$ for any $\varepsilon >0,$ and any $v\in \nu^m,$ 
\begin{align*}
&\left|\left\langle \nabla^{(2)}P_{\tau}^m(f_N(v+\varepsilon h_a^m))-\nabla^{(2)}P_{\tau}^m(f_N(v )),h_a^m\otimes h^m_a\right\rangle_{H\otimes H}\right| \\
&\leq c \|v\|_{C([0,1],{\mathbb R}^d)}\frac{e^{-\frac{5\tau}{2}}}{\beta^2_{\frac{\tau}{2}}} \varepsilon N^{-\frac{1}{2}} \sqrt{\frac{N^3}{m^3}}.
\end{align*}
\end{lemma}

\begin{proof}
From equality (4.9) in \cite{CouDec2020} we have
\begin{align}\label{4.9}
&\left( \frac{e^{-\frac{3}{2}\tau}}{\beta_{\frac{\tau}{2}}^2}\right)^{-1}\left\langle\nabla^{(2)}P_{\tau}^m(f_N(v )),h\otimes h\right\rangle_{H\otimes H}=\nonumber\\
&{\mathbb E} \left( f_N\left( w_{\tau}(\pi^N(v),\pi^N(B^m),\pi^N(\hat{B}^m))\right){\mathbb E} \left(\delta h(B^m) \left|\right.\pi^N(B^m)\right) {\mathbb E} \left(\delta h\hat{B}^m\left|\right.\pi^N(\hat{B}^m)\right)\right),
\end{align}
where $w_{\tau}(v,y,z)=e^{-\frac{\tau}{2}}\left( e^{-\frac{\tau}{2}}v + \beta_{\tau/2}y\right)+\beta_{\tau/2}z$
and $\hat{B}^m$ is an independent copy of $B^m.$

On the other hand for Lemma 4.7 of \cite{CouDec2020}, we know that for $m>8N$
\begin{align}\label{maj-var}
Var \left( {\mathbb E} \left(\delta h(\hat{B}^m)\left|\right.\pi^N(\hat{B}^m)\right)\right) \leq c\frac{N}{m},
\end{align}
and the same holds for the other conditional expectation. Using Cauchy-Schwarz inequality in \eqref{4.9} and taking \eqref{maj-var} into account we get
\begin{align*}
&\left( \frac{e^{-\frac{3}{2}\tau}}{\beta_{\tau/2}^2}\right)^{-1}\left|\left\langle \nabla^{(2)}P_{\tau}^m(f_N(v +\varepsilon h_a^m))- \nabla^{(2)}P_{\tau}^m(f_N(v )),h_a^m\otimes h_a^m\right\rangle_{H\otimes H}\right|\\
& \leq c\left(\frac{N}{m}\right)^2{\mathbb E} \left(\|\pi^N(v)+\pi^N(B^m) +\pi^N(\hat{B}^m)\|_{C([0,1],{\mathbb R}^d)}\right)\varepsilon e^{-\tau} \|\pi^N(h_a^m)\|
_{C([0,1],{\mathbb R}^d)}\\
&\quad+c\left(\frac{N}{m}\right)^2\varepsilon e^{-\tau} \|\pi^N(h_a^m)\|
_{C([0,1],{\mathbb R}^d)}.
\end{align*}

According to Theorem 3.1 of \cite{CouDec2020},
\begin{align*}
\sup_{N,m} {\mathbb E} \left( \|\pi^N(B^m)\|_{C([0,1],{\mathbb R}^d)}\right)<\infty,
\end{align*}
and 
\begin{align}\label{maj-4.11}
\left( \frac{e^{-\frac{3}{2}\tau}}{\beta_{\tau/2}^2}\right)^{-1}\left|\left\langle \nabla^{(2)}P_{\tau}^m(f_N(v +\varepsilon h_a^m))- \nabla^{(2)}P_{\tau}^m(f_N(v )),h_a^m\otimes h_a^m\right\rangle_{H\otimes H}\right| \\\notag
\leq c\left(\frac{N}{m}\right)^2\left(1+\|v\|_{C([0,1],{\mathbb R}^d)}\right)\varepsilon e^{-\tau} \|\pi^N(h_a^m)\|
_{C([0,1],{\mathbb R}^d)}.
\end{align}
Note that $ \|h_a^N\|
_{C([0,1],{\mathbb R}^d)}=1/\sqrt{N}.$
Using the same lines in the end of the proof of Theorem 4.4 of \cite{CouDec2020} where inequality (4.11) of \cite{CouDec2020} is replaced by \eqref{maj-4.11}, we conclude that
\begin{equation*}
\|\pi^N(h_a^m)\|
_{C([0,1],{\mathbb R}^d)}\leq c\sqrt{\frac{N}{m}}N^{-1/2},
\end{equation*}
and we achieve the proof of Lemma \ref{lem-thme4.4}.
\end{proof}

According to Lemmas \ref{lem4.3} and \ref{lem-thme4.4}, we obtain the following result.

\begin{proposition}\label{thme-4.5}
Let $p\geq 4.$ If $X\in L^p$ then there exists a positive constant $c$ such that for all $\tau >0$ and $F \in {\mathcal Lip}_{1,1}(C([0,T],{\mathbb R}^d))$, we have
\begin{align*}
{\mathbb E}\left(\int_{\tau}^{\infty}LP_u(f_N(S^m)) du \right) \leq c\|X\|^4_{L^p} \frac{N}{\sqrt{m}}\int_{\tau}^{\infty}\frac{e^{-\frac{5}{2}u}}{\beta_{\frac{u}{2}}^2}du.
\end{align*}
\end{proposition}

\begin{proof}
According to Lemmas \ref{lem4.3} and \ref{lem-thme4.4}, since the cardinality of ${\mathcal A}^m$ is $m\times d$ we obtain
\begin{align*}
{\mathbb E}\left(\int_{\tau}^{\infty}LP_u(f_N(S^m)) du \right) \leq c{\mathbb E}\left((1+\|X\|^2)\|X\|\|S^m\|_{C([0,1],{\mathbb R}^d)}\right) \frac{N}{\sqrt{m}}\int_{\tau}^{\infty}\frac{e^{-\frac{5}{2}u}}{\beta_{\frac{u}{2}}^2}du.
\end{align*}

Using the result established in Theorem \ref{cont-proj}, we have that 
\begin{equation*}
{\mathbb E}\left(\|S^m\|^4_{C([0,1],{\mathbb R}^d)}\right)^{1/4}\leq c\|X\|_{L^4}.
\end{equation*} 
Therefore, applying H\"older inequality for $q=\frac{4}{3}$ and its conjugate exponent $q'=4$ we obtain 
\begin{align*}
{\mathbb E}\left(\int_{\tau}^{\infty}LP_u(f_N(S^m)) du \right) \leq c\|X\|_{L^p}^4 \frac{N}{\sqrt{m}}\int_{\tau}^{\infty}\frac{e^{-\frac{5}{2}u}}{\beta_{\frac{u}{2}}^2}du.
\end{align*}
\end{proof}

Finally, the estimate of the term $I_2$ is obtained by combining Lemma \ref{est-0-tau} and Proposition \ref{thme-4.5} in equality \eqref{s-aux-ec}, i.e.,
\begin{align*}
{\mathbb E}\left(f_N(S^m)\right) - {\mathbb E}\left(f_N(B^m)\right)\leq c\|X\|_{L^p}^4\left( \sqrt{1-e^{-2\tau}} + \frac{N}{\sqrt{m}}\int_{\tau}^{\infty}\frac{e^{-\frac{5}{2}u}}{\beta_{\frac{u}{2}}^2}du\right).
\end{align*}
Optimizing with respect to $\tau$ we obtain Proposition \ref{prop-thme3.3}.

%%%%%%%%%%%%%%%%%%%%%%%%%%%%%%%%%%%%%%%%%%%%%%%%%%%%%%%%%%%%%%%%%%%%%%%%%%%%%%%%%%%%%
%%%%%%%%%%%%%%%%%%%%%%%%%%%%%%%%%%%%%%%%%%%%%%%
%%%%%%%%%%%%  Kac's counting formula  %%%%%%%%%%%%%%%%%%%%%%%%%%%%%%%%%%%%%%%%%%%%
\section{Kac's counting formula}\label{kac}
Let $f\,:[t_1,t_2]\to \mathbb{R}$ be a real-valued function defined in the interval $[t_1,t_2]$ of the real line. We denote 
\begin{equation*}
Z\left(f,[t_1,t_2]\right)=\{t\in [t_1,t_2]\,: f(t)=0\},
\end{equation*}
the set of the roots of equation $f(t)=0$ in the interval $[t_1,t_2]$ and 
\begin{equation}\label{N}
N(f,[t_1,t_2])=\#Z\left(f,[t_1,t_2]\right)
\end{equation}
the number of these roots. 

\begin{hyp}\label{hyp-H1} 
We say that a function $f~:[t_1,t_2]\to \mathbb{R}$ fulfills Hypothesis \ref{hyp-H1} if and only if
\begin{enumerate}
\item  $f$ is $C^1,$
\item  $f(t_1)f(t_2)\neq 0,$
\item $\{t \in [t_1,t_2],~~f(t)=0 ~~\mbox{  and  }~~\dot{f}(t)=0\}=\emptyset.$
\end{enumerate}
\end{hyp}
%%%%%%%%%%%%%%%%%%%%%%%%%%%%%%%%%%%%%%%%%%%%%%%%%%%%%%%%
A careful reading of the proof of \cite[Lemma 3.1]{Wse-Aza} allows to obtain the following Lemma.

\begin{lemma}\label{lem-n-delta-sta}
Let $f$ be a function such that satisfies Hypothesis \ref{hyp-H1}, then
\begin{equation*}\min_{t\in [t_1,t_2]}|f(t)|+|\dot{f}(t)| >0,
\end{equation*} and 
$N(f,[t_1,t_2])<\infty$. Moreover, if 
\begin{equation*}
0<\delta< \min \left( |f(t_1)|,|f(t_2)|,\frac{1}{2} \min_{t\in (t_1,t_2)}|f(t)|+|\dot{f}(t)|\right),
\end{equation*}
then
\begin{align*}\label{number-zero-sta}
N(f,[t_1,t_2])= \frac{1}{2 \delta} \int_{t_1}^{t_2} |\dot{f}(t)|{\mathbf 1}_{|f(t)|\leq \delta}dt.
\end{align*}
\end{lemma}
%%%%%%%%%%%%%%%%%%%%%%

\begin{proof} 

\

{\bf Step 1~:} The map $t \mapsto |f(t)| +|\dot{f}(t)|$ is continuous on the compact interval $[t_1,t_2]$, then it reaches its infimum. That is to say, there exists $t_0\in [t_1,t_2]$ such that
$$0<|f(t_0)| +|\dot{f}(t_0)| = \min_{t\in [t_1,t_2]}|f(t)|+|\dot{f}(t)|.$$
The inequality in the left member follows from point 3 of Hypothesis \ref{hyp-H1}.

\

{\bf Setp 2~:} If $f$ fulfills Hypothesis \ref{hyp-H1}, $N(f,[t_1,t_2])<\infty$.\\
Assume that there exists $\{t_n\}_n$ a sequence of points of $[t_1,t_2]$ such that $t_n \neq t_m,~~\forall n\neq m,$ and $f(t_n)=0,~~\forall n.$ Up to extracting a subsequence, we can assume that $\{t_n\}_n$ converges to $t_{\infty} \in [t_1,t_2]$ and since $f$ is continuous $f(t_{\infty})=0.$ 

On the other hand, for all $n,$ there exists $s_n \in [t_n,t_{n+1}]$ (or $[t_{n-1},t_n])$ such that
$$ \dot{f}(s_n)= \frac{f(t_{n+1}) -f(t_n)}{t_{n+1}-t_n}=0.$$
The sequence $\{s_n\}$ converges to $t_{\infty}$ and since $\dot{f}$ is continuous $\dot{f}(t_{\infty})=0.$
Then $t_{\infty} \in \{t \in [t_1,t_2],~~f(t)=0 ~~\mbox{  and  }~~\dot{f}(t)=0\}=\emptyset.$ This contradicts point 3 of Hypothesis \ref{hyp-H1}.

In consequence, we have proved that if $f$ fulfills Hypothesis \ref{hyp-H1}, the set of zeros of $f$ in $[t_1,t_2]$ is finite.

\

{\bf Step 3~:} Assume that $f(t)\neq 0,~~\forall t \in [t_1,t_2]$ and let $\delta$
$$0< \delta < \min \left( |f(t_1)|,|f(t_2)|,\frac{1}{2} \min_{t\in [t_1,t_2]}|f(t)|+|\dot{f}(t)|\right).$$
Then $L_{\delta}:= \{t\in [t_1,t_2]\,: |f(t)| \leq \delta\}= \emptyset$ and $$N(f,[t_1,t_2])= \frac{1}{2 \delta} \int_{t_1}^{t_2} |\dot{f}(t)|{\mathbf 1}_{|f(t)|\leq \delta}dt.$$
Indeed, assume that $L_{\delta}\neq \emptyset $ and let $|a,b|$ be a connected component of $L_{\delta}.$
Since $f$ is continuous and from the definition of $L_\delta,$ $t_1 <a\leq b < t_2$ and  $|f(a)|=|f(b)|=\delta.$ In addition, $f(t)\neq 0$ for all $t\in[t_1,t_2]$ therefore the only possibility is that $f(a)=f(b)=\delta$ thus, there exists $c \in [a,b]$ such that $\dot{f}(c)=0$ and $\delta <|f(c)| \leq \delta$  which contradicts the choice of $\delta$. Therefore $L_{\delta}=\emptyset$.
%%%%%%%%%%%%%%%%%%%%%%%%%%%%%%%%%%%%%%%%
~~~\\

%%%%%%%%%%%%%%%%%%%%%%%%%%%%%%%%
{\bf Step 4~:} Assume that $\exists t\in [t_1,t_2],~~f(t)=0$. Let $\{s_1,...,s_n\}$ be the set of zeros of $f$ in $[t_1,t_2]$ and let $$0< \delta < \min \left( |f(t_1)|,|f(t_2)|,\frac{1}{2} \min_{t\in [t_1,t_2]}|f(t)|+|\dot{f}(t)|\right).$$
Then, the connected component of $s_k$ in $L_{\delta}:= \{t\in [t_1,t_2]\,: |f(t)| \leq \delta\}$ are disjoints and they are the only one.
Indeed, let $|a,b|$ be a connexe component of $L_{\delta}.$ From the definition of $\delta,$ the fact that $f$ is continuous and point 3 in Hypothesis \ref{hyp-H1}, $|f(a)|=|f(b)|=\delta$ and $f(a)f(b) <0.$ Then, $f$ has a zero in $|a,b|$ namely $s.$ Moreover, $f$ is a $C^1$ diffeomorphism on $]a,b[$ into $f(]a,b[).$ Then, $s$ is the only zero of $f$ in $|a,b|.$

~~\\

{\bf Step 5~:} Assume that $f$ has $n$ zeros in $[t_1,t_2].$ Let $$0<\delta < \min \left( |f(t_1)|,|f(t_2)|,\frac{1}{2} \min_{t\in [t_1,t_2]}|f(t)|+|\dot{f}(t)|\right).$$
Then
$$L_{\delta}=\{t\in [t_1,t_2],~~|f(t)|\leq \delta\}=\cup_{k=1}^n |a_k,b_k|,$$
where the union is disjoint and
\begin{align*}
 \frac{1}{2 \delta} \int_{t_1}^{t_2} |\dot{f}(t)|{\mathbf 1}_{|f(t)|\leq \delta}dt=\sum_{k=1}^n\frac{1}{2 \delta} \int_{a_k}^{b_k} |\dot{f}(t)|{\mathbf 1}_{|f(t)|\leq \delta}dt.
\end{align*}
On $|a_k,b_k|$ we perform the change of variable $u=f(t)$ and
\begin{align*}
 \frac{1}{2 \delta} \int_{t_1}^{t_2} |\dot{f}(t)|{\mathbf 1}_{|f(t)|}dt=n=N(f,[t_1,t_2]).
\end{align*}

\end{proof}

%%%%%%%%%%%%%%%%%%%%%%%%%%%%%%%%%%%%%%%%%%%%%%%%%%%%%

%%%%%%%%%%%%%%%%%%%%%%%%%%%%%%%%%%%%%%%%%%%%%%%%%%%%%%%%%%
%%%%%%%%%%%%%%% BIBLIOGRAPHY %%%%%%%%%%%%%%%%%%%%%%%%%%%%%%%%%%
%%%%%%%%%%%%%%%%%%%%%%%%%%%%%%%%%%%%%%%%%%%%%%%%%%%%%%%%%%

\bibliographystyle{siam}
\bibliography{references}

\begin{thebibliography}{10}

\bibitem{angst}
{\sc J.~Angst, V.-H. Pham, and G.~Poly}, {\em Universality of the nodal length
  of bivariate random trigonometric polynomials}, Trans. Amer. Math. Soc., 370
  (2018), pp.~8331--8357.

\bibitem{JMA-Leon}
{\sc J.-M. Aza\"{\i}s, F.~Dalmao, J.~R. Le\'on, I.~Nourdin, and G.~Poly}, {\em
  Local universality of the number of zeros of random trigonometric polynomials
  with continuous coefficients}, Preprint available at:,
  \url{https://arxiv.org/pdf/1512.05583.pdf} (2018).

\bibitem{Azais-Leon}
{\sc J.-M. Aza\"{\i}s and J.~R. Le\'{o}n}, {\em C{LT} for crossings of random
  trigonometric polynomials}, Electron. J. Probab., 18 (2013), pp.~no. 68, 17.

\bibitem{Wse-Aza}
{\sc J.-M. Aza\"{\i}s and M.~Wschebor}, {\em Level sets and extrema of random
  processes and fields}, John Wiley \& Sons, Inc., Hoboken, NJ, 2009.

\bibitem{Barbour90}
{\sc A.~D. Barbour}, {\em Stein's method for diffusion approximations}, Probab.
  Theory Related Fields, 84 (1990), pp.~297--322.

\bibitem{Bogo}
{\sc E.~Bogomolny, O.~Bohigas, and P.~Leboeuf}, {\em Quantum chaotic dynamics
  and random polynomials}, J. Statist. Phys., 85 (1996), pp.~639--679.

\bibitem{CoutinDecre}
{\sc L.~Coutin and L.~Decreusefond}, {\em Stein's method for {B}rownian
  approximations}, Commun. Stoch. Anal., 7 (2013), pp.~349--372.

\bibitem{CouDec2020}
{\sc L.~Coutin and L.~Decreusefond}, {\em Donsker's theorem in wasserstein-1
  distance}, Electron. Commun. Probab., 25 (2020), pp.~1--13.

\bibitem{decreu}
{\sc L.~Decreusefond}, {\em The {S}tein-{D}irichlet-{M}alliavin method}, in
  Mod\'{e}lisation {A}l\'{e}atoire et {S}tatistique---{J}ourn\'{e}es {MAS}
  2014, vol.~51 of ESAIM Proc. Surveys, EDP Sci., Les Ulis, 2015, pp.~49--59.

\bibitem{Dunna}
{\sc J.~E.~A. Dunnage}, {\em The number of real zeros of a random trigonometric
  polynomial}, Proc. London Math. Soc. (3), 16 (1966), pp.~53--84.

\bibitem{Dunnage}
\leavevmode\vrule height 2pt depth -1.6pt width 23pt, {\em The number of real
  zeros of a random trigonometric polynomial}, Proc. London Math. Soc. (3), 16
  (1966), pp.~53--84.

\bibitem{erdos}
{\sc P.~Erd\"{o}s and A.~C. Offord}, {\em On the number of real roots of a
  random algebraic equation}, Proc. London Math. Soc. (3), 6 (1956),
  pp.~139--160.

\bibitem{Farahmand87}
{\sc K.~Farahmand}, {\em On the number of real zeros of a random trigonometric
  polynomial: coefficients with nonzero infinite mean}, Stochastic Anal. Appl.,
  5 (1987), pp.~379--386.

\bibitem{Farahmand95}
\leavevmode\vrule height 2pt depth -1.6pt width 23pt, {\em Level crossings of a
  random trigonometric polynomial with dependent coefficients}, J. Austral.
  Math. Soc. Ser. A, 58 (1995), pp.~39--46.

\bibitem{Farahmand97}
\leavevmode\vrule height 2pt depth -1.6pt width 23pt, {\em On the variance of
  the number of real zeros of a random trigonometric polynomial}, J. Appl.
  Math. Stochastic Anal., 10 (1997), pp.~57--66.

\bibitem{FarahSamba}
{\sc K.~Farahmand and M.~Sambandham}, {\em On the expected number of real zeros
  of random trigonometric polynomials}, Analysis, 17 (1997), pp.~345--353.

\bibitem{FlasHen}
{\sc H.~Flasche}, {\em Expected number of real roots of random trigonometric
  polynomials}, Stochastic Process. Appl., 127 (2017), pp.~3928--3942.

\bibitem{friz}
{\sc P.~K. Friz and N.~B. Victoir}, {\em Multidimensional stochastic processes
  as rough paths}, vol.~120 of Cambridge Studies in Advanced Mathematics,
  Cambridge University Press, Cambridge, 2010.
\newblock Theory and applications.

\bibitem{Iks2016}
{\sc A.~Iksanov, Z.~Kabluchko, and A.~Marynych}, {\em Local universality for
  real roots of random trigonometric polynomials}, Electron. J. Probab., 21
  (2016), pp.~Paper No. 63, 19.

\bibitem{Nisio}
{\sc K.~It\^{o} and M.~Nisio}, {\em On the convergence of sums of independent
  {B}anach space valued random variables}, Osaka Math. J., 5 (1968),
  pp.~35--48.

\bibitem{jacod}
{\sc J.~Jacod}, {\em Calcul stochastique et probl\`emes de martingales},
  vol.~714 of Lecture Notes in Mathematics, Springer, Berlin, 1979.

\bibitem{kac}
{\sc M.~Kac}, {\em On the average number of real roots of a random algebraic
  equation. {II}}, Proc. London Math. Soc. (2), 50 (1949), pp.~390--408.

\bibitem{littlewood}
{\sc J.~E. Littlewood and A.~C. Offord}, {\em On the {N}umber of {R}eal {R}oots
  of a {R}andom {A}lgebraic {E}quation}, J. London Math. Soc., 13 (1938),
  pp.~288--295.

\bibitem{ustunel}
{\sc A.~S. \"{U}st\"{u}nel}, {\em An introduction to analysis on {W}iener
  space}, vol.~1610 of Lecture Notes in Mathematics, Springer-Verlag, Berlin,
  1995.

\bibitem{villani}
{\sc C.~Villani}, {\em Optimal transport: old and new}, vol.~338, Springer
  Science \& Business Media, 2008.

\end{thebibliography}
\end{document}